\documentclass{amsart}

\usepackage[pdfborder={0 0 0}, draft=false, pagebackref, unicode=true]{hyperref}

\def\corkscrew{\varkappa}

\usepackage{esint,amssymb,mathrsfs,mathtools,amsthm,tikz}
\usetikzlibrary{shadings}

\def\multichoose#1#2{\ensuremath{\left(\kern-.3em\left(\genfrac{}{}{0pt}{}{#1}{#2}\right)\kern-.3em\right)}}

\newtheorem{thm}[equation]{Theorem} 
\newtheorem{lem}[equation]{Lemma}
\newtheorem{cor}[equation]{Corollary}
\newtheorem{prp}[equation]{Proposition}

\theoremstyle{definition}
\newtheorem{defn}[equation]{Definition}

\newtheorem{rmk}[equation]{Remark}

\numberwithin{equation}{section}

\usepackage{cleveref}
\crefname{equation}{}{}
\crefname{lem}{}{}
\crefrangelabelformat{equation}{(#3#1#4)--(#5#2#6)}

\newcommand\abs[2][empty]{\csname#1\endcsname \lvert{#2}\csname#1\endcsname\rvert}
\newcommand\doublebar[2][empty]{\csname#1\endcsname \lVert{#2}\csname#1\endcsname\rVert}

\newcommand\dist{\mathop{\mathrm{dist}}\nolimits}

\newcommand\Div{\mathop{\mathrm{div}}\nolimits}
\newcommand\Tr{\mathop{\mathrm{Tr}}\nolimits}

\newcommand\supp{\mathop{\mathrm{supp}}\nolimits}
\newcommand\diam{\mathop{\mathrm{diam}}\nolimits}

\newcommand\esssup{\mathop{\mathrm{ess\,sup}}}

\newcommand\R{\mathbb{R}}
\newcommand\Z{\mathbb{Z}}

\newcommand\dmn{{\mathfrak{n}}}  
\newcommand\pdmn{{\mathfrak{n}}} 
\newcommand\bdmn{{\mathfrak{d}}} 
\newcommand\dmnMinusOne{{\mathfrak{n}-1}} 

\newcommand\pbdmn{{\mathfrak{d}}}
\newcommand\bpdmn{{\mathfrak{d}}}

\makeatletter\def\HyPsd@CatcodeWarning#1{}\makeatother

\title[The Besov Dirichlet problem]{The Dirichlet problem with boundary data of fractional smoothness in domains with Ahlfors regular boundary}

\author{Ariel Barton, Svitlana Mayboroda, Alberto Pacati}
\thanks{
Svitlana Mayboroda is partially supported by the NSF RAISE DMS 1839077 and the Simons Initiative of Geometry of Flows. Svitlana Mayboroda and Alberto Pacati are both partially supported by the Simons Collaboration on Localisation of Waves, 563916 SM}
\subjclass[2020]{35J25, 46E35}
\keywords{Dirichlet, Besov spaces, Ahlfors regular, extension} 
\address{
AB: Department of Mathematical Sciences, University of Arkansas, Fayetteville, Arkansas, 72701, USA.}

\address{
SM: Department of Mathematics, R\"amistrasse 101, ETH Z\"urich, 8092 Z\"urich, Switzerland.
}
\address{
AP: Department of Mathematics, R\"amistrasse 101, ETH Z\"urich, 8092 Z\"urich, Switzerland.
}

\email{aeb019@uark.edu}
\email{svitlana.mayboroda@math.ethz.ch} 
\email{alberto.pacati@math.ethz.ch}
\date{\today}

\begin{document}
\begin{abstract} We establish well posedness of the Dirichlet problem in weak local John domains $\Omega\subset\R^\dmn$ with Ahlfors regular boundaries of possibly  non integer dimension $0<\bdmn<\dmn$, for linear  second order elliptic equations with real coefficients, and with data in Besov spaces with a very broad range of acceptable parameters.
\end{abstract}

\maketitle

\setcounter{tocdepth}{2}

\tableofcontents

\setcounter{tocdepth}{3}

\clearpage
\section{Introduction}
We consider the Poisson-Dirichlet problem
\begin{equation}\label{eqn:one}
\left\{\begin{aligned}
-\Div(A\nabla u)&=h&&\text{in }\Omega,\\
u&=f&&\text{on }\partial\Omega,
\end{aligned}\right. \end{equation}
where $A$ is a real uniformly elliptic matrix-valued function, $\Omega\subset\R^\dmn$ is an open set that satisfies suitable openness and connectedness assumptions, and $\partial\Omega$ is $\bdmn$-Ahlfors regular, for some $0<\bdmn<\dmn$. Well posedness of~\eqref{eqn:one} with data $h$ and $f$ in various spaces has received substantial attention in the last several decades. By now, the sharp range of coefficients $A$ and domains $\Omega$ for which the $L^p$-Dirichlet problem (corresponding to $h=0$ and $f\in L^p(\partial\Omega)$) is well posed is rather well understood (see, e.g., \cite{DavJ90, JerK82, HofM14, AzzHMMT20, DavM23, Fen22, KenP01, CavHMT20, FenLM24}). 
Great progress has also been made in understanding the regularity problem, that is, the Dirichlet problem with boundary data $f$ lying in a Sobolev space $\dot W^{1,p}(\partial\Omega)$  (see, e.g., \cite{KenP93, Shen2019, DindosHofmannPipher2023, MouPT22p}). While the $L^p$-Dirichlet problem corresponds to smoothness zero of the boundary data, the regularity problem corresponds to smoothness one.

In between these two endpoints are spaces of fractional smoothness, the so-called Besov spaces $\dot B^{s,p}_q(\partial\Omega)$. For the Laplacian in Lipschitz domains, in \cite{JerK95, MayMit04A}, the authors identified the sharp range of parameters for which~\eqref{eqn:one} is  well posed with respect to the requirements 
\begin{gather*}
f\in \dot B^{s,p}_q(\partial\Omega),\quad h\in \dot B^{s+1/p-2,p}_q(\Omega),\\
u\in \dot B^{s+1/p,p}_q(\Omega).
\end{gather*}
The regularity $u\in \dot B^{s+1/p,p}_q(\Omega)$ is natural because any function $f$ in $\dot B^{s,p}_q(\partial\Omega)$ has an extension $Ef$ in $\dot B^{s+1/p,p}_q(\Omega)$ (see \cite[Chapter 5]{JonW84}). The regularity $h\in \dot B^{s+1/p-2,p}_q(\Omega)$ is natural because the Laplacian $\Delta u$ of a function $u$ in $\dot B^{s+1/p,p}_q(\Omega)$ must satisfy $\Delta u\in \dot B^{s+1/p-2,p}_q(\Omega)$ (see \cite[Chapter 2]{Tri83}).

When considering more general coefficients $A$, we cannot expect such a high regularity of the solution $u$.
Specifically, if $s+1/p>1$ then $u\in \dot B^{s+1/p,p}_q(\Omega)$ implies that $\nabla u\in \dot B^{s+1/p-1,p}_q(\Omega)$, a space of functions with smoothness properties considerably beyond local integrability. But in general, solutions to uniformly elliptic equations do not have a gradient with a positive degree of smoothness (see the introduction of \cite{BarM16A}). We instead expect solutions $u$ with one degree of smoothness.
A suitable environment for the solution to live in is an appropriate {\it weighted} Sobolev space $\dot {W} ^ {1,p}_{w_{p,s}}$, as seen in the works \cite{MitT06, MazMS10} (see also \cite{DiF96, KinZ01, AusQ02}). They establish  well posedness of the Poisson-Dirichlet problem~\eqref{eqn:one} in Lipschitz and $C^1$ domains for matrices $A$ satisfying VMO or Dini-type conditions with respect to the requirements
\begin{equation}\label{eqn:unaveraged}
\begin{gathered}
f\in \dot B^{s,p}_p(\partial\Omega), \quad h=\Div\vec H, \quad \vec H\in L^p_{w_{p,s}}(\Omega),\\
u\in\dot W^{1,p}_{w_{p,s}}(\Omega),
\end{gathered}\end{equation}
for a certain range of parameters $1<p<\infty$ and $0<s<1$. The weight $w_{p,s}$ is defined by
\begin{equation}\label{eqn:weight:one}
dw_{p,s}=\dist (\cdot,\partial\Omega)^{-1+p-ps}\,dx.
\end{equation}
These spaces are sharp in Lipschitz domains, in the sense that boundary values (that we refer to as traces) of functions in $\dot W^{1,p}_{w_{p,s}}(\Omega)$ lie in $\dot B^{s,p}_p(\partial\Omega)$ (see \cite{MitT06, Kim07}) and functions in $\dot B^{s,p}_p(\partial\Omega)$ admit an extension to $\dot W^{1,p}_{w_{p,s}}(\Omega)$ (see  \cite{Usp61, Nik77B, Sha85, NikLM88, MitT06, Kim07, MazMS10}).

In \cite{BarM16A}, the authors (two of the authors of the current manuscript) studied well posedness of the problem~\eqref{eqn:one} in Lipschitz domains for matrices $A$ with coefficients independent of a direction transverse to the boundary (see also \cite{AmeA18}). In such generality of coefficients, the gradient of solutions does not have high local integrability; it is only in $L^\beta_{loc}(\Omega)$ for some $\beta$ possibly close to $2$. Therefore \cite{BarM16A} introduced local averaging so as not to affect the range of parameters $s, p$ for which there is  well posedness of the problem~\eqref{eqn:one} with boundary data $f\in \dot B^{s,p}_p(\partial\Omega)$. Namely, they generalized the weighted Sobolev spaces $\dot W^{1,p}_{w_{p,s}}$, considered in~\eqref{eqn:unaveraged}, to $\dot W^{1,p}_{av,\beta,s}$, where the norm is defined by
\begin{equation}\label{eqn:averaged:lip}
\|u\|_{\dot W^{1,p}_{av,\beta,s}(\Omega)}=\biggl(
\int_\Omega
\biggl(\fint_{B(x,\dist(x,\partial\Omega)/2)} |\nabla u|^\beta \biggr)^{p/\beta}
\dist(x,\partial\Omega)^ {-1+p-ps}
\, dx\biggr)^{1/p}.
\end{equation}
In such spaces, the authors obtained well posedness for a sharp range of parameters and proved trace bounds that yield sharpness of the norms. The inspiration for taking local averages in~\eqref{eqn:averaged:lip} comes from the definition of the modified nontangential maximal function $\tilde N_+(\nabla u)$ given in \cite{KenP93}. 

Beyond Lipschitz domains, it was recently shown in \cite{CaoHMPZ25} that the Dirichlet problem with Hölder continuous ($\dot C^s(\partial\Omega)$) boundary data $f$ and solution $u\in\dot C^s(\Omega)$ is well posed (for $s$ sufficiently small) in domains with boundaries of codimension $1$ satisfying a Wiener-type criterion, and for general operators with elliptic and uniformly bounded coefficients. In our terminology, $\dot C^s(\partial\Omega)=\dot B^{s,\infty}_\infty(\partial\Omega)$, and it is therefore a space of fractional smoothness. For boundaries of higher or mixed codimensions, the existence and uniqueness of energy solutions in $\dot W^{1,2}(\Omega)$, corresponding to boundary data in the Besov space $\dot B^{1/2,2}_2(\partial\Omega)$, were proven in \cite{DavFM21} and \cite{DavFM23B}.
We refer the reader to \cite{BarMP25pD} for a more detailed survey of results about the Poisson and Dirichlet problem.   

With the present work and its companion \cite{BarMP26pA}, we prove  well posedness of the Dirichlet-Poisson problem for a wide range of parameters in very general domains with Ahlfors regular boundaries of any (possibly non-integer) dimension $0<\bdmn<\dmn$ and in the full generality of real elliptic uniformly bounded coefficients. 
Moreover, we provide an extrapolation criterion, improving the range of parameters for domains and coefficients for which the $L^q$-Dirichlet problem is solvable for some $1<q<\infty$.

Our first main theorem is as follows.
\begin{thm}\label{thm:Dirichlet}
Let $\Omega\subseteq\R^\dmn$, $\dmn\geq 2$, be a connected open set. Suppose that $\partial\Omega$ is $\bdmn$-Ahlfors regular for some real number $\bdmn$ with $0<\bdmn<\dmn$, and that $\R^\dmn\setminus\Omega$ is unbounded.

If $\bdmn=\dmnMinusOne$, we additionally require that $\Omega$ satisfies the weak local John condition. If $\dmnMinusOne<\bdmn<\dmn$, we require $\Omega$ to satisfy the interior corkscrew and Harnack chain conditions.

Let $L=-\Div A\nabla$ be a linear differential operator associated to real, not necessarily symmetric coefficients $A:\Omega\mapsto\R^{\dmn\times\dmn}$ that satisfy the positive definiteness and boundedness conditions (\ref{eqn:elliptic:introduction}--\ref{eqn:elliptic:introduction:1}).

There are then numbers $\mathfrak{a}$, $\mathfrak{a}^*\in (0,1]$ and $p^-$, $p^+$ with $p^-<2<p^+$ such that, if $(s,1/p)$ lies in the open pentagon indicated in Figure~\ref{fig:arbitrary}, and if $p^-<\beta<p^+$, then
for all $f\in \dot B^{s,p}_p(\partial\Omega)$ (see Definition~\ref{dfn:Besov}) and all
$\vec H:\Omega\mapsto\R^\dmn$ such that the right-hand side of the bound~\eqref{eqn:Besov:estimate} is finite, there is a unique solution $u$ to the Poisson-Dirichlet problem
\begin{equation}\label{eqn:Dirichlet}
\left\{\begin{aligned}
Lu&=-\Div(\delta^{\bdmn+1-\pdmn}\vec H)&&\text{in }\Omega,\\
u&=f&&\text{on }\partial\Omega\text{, in the sense of Definition~\ref{dfn:trace} below,}
\end{aligned}\right.\end{equation}
that satisfies
\begin{multline}\label{eqn:Besov:estimate}
\int_\Omega
\biggl(\fint_{B(x,\delta(x)/2)} |\nabla u|^\beta \biggr)^{p/\beta}
\delta(x)^ {\bdmn-\pdmn+p-ps}
\,dx
\\\leq
C\|f\|_{\dot B^{s,p}_p(\partial\Omega)}^p
+C\int_\Omega
\biggl(\fint_{B(x,\delta(x)/2)}|\vec H|^\beta \biggr)^{p/\beta}
\delta(x)^ {\bdmn-\pdmn+p-ps}
\,dx
\end{multline}
for some constant $C$ depending on $p$, $s$, $\beta$, $A$, and~$\Omega$, but not on~$\vec H$ or~$f$.

If $0<s<\mathfrak{a}$, $p^-<\beta<p^+$, and the right-hand side of the bound~\eqref{eqn:Besov:estimate:infinity} is finite, then there is a unique solution to the problem~\eqref{eqn:Dirichlet} that satisfies the estimate
\begin{equation}\label{eqn:Besov:estimate:infinity}
\sup_{x\in\Omega}
\biggl(\fint_{B(x,\delta(x)/2)}|\nabla u|^\beta\biggr)^{1/\beta}
\delta(x)^ {1-s}
\leq C\|f\|_{\dot C^s(\partial\Omega)}+
C\sup_{x\in\Omega}
\biggl(\fint_{B(x,\delta(x)/2)}|\vec H|^\beta\biggr)^{1/\beta}
\delta(x)^ {1-s}
.\end{equation}

Furthermore, we have that $\mathfrak{a}\geq \max(\alpha,1-\bdmn)$ and $\mathfrak{a}^*\geq \max(\alpha^*,1-\bdmn)$, where $\alpha$ is as in the boundary De Giorgi-Nash estimate (Lemma~\ref{lem:boundary:DGN}), and $\alpha^*$ is as in Lemma~\ref{lem:boundary:DGN} with $L$ replaced by~$L^*$.

Finally, $p^+$ may be taken to be the number $p^+_L$ in Meyers's interior reverse Hölder estimate \eqref{eqn:interior:Meyers} for~$L$, and $p^-$ may be taken to be the Hölder conjugate of the number $p^+_{L^*}$ in Meyers's reverse Hölder estimate for~$L^*$.
\end{thm}

\def\figurealpha{0.3}
\def\figuredimen{2}
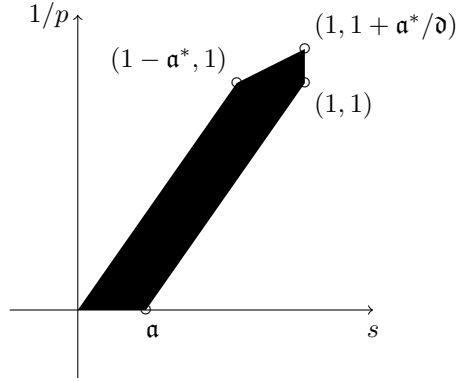
\begin{figure}
\begin{tikzpicture}[scale=3]
\begin{scope}[shift={(0,0)}]
\draw [->] (-0.3,0)--(1.3,0) node [below ] {$\vphantom{1}s$};
\draw [->] (0,-0.3)--(0,1.3) node [ left] {$1/p$};
\fill  (1,1) -- (\figurealpha,0) -- (0,0) -- (1-\figurealpha,1) --(1,1+\figurealpha/\figuredimen) -- cycle;
\node at (1,1) {$\circ$};
\node [below right] at (1,1) {$(1,1)$};
\node at (1,1+\figurealpha/\figuredimen) {$\circ$};
\node [above right] at (1,1+\figurealpha/\figuredimen) {$(1,1+\mathfrak{a}^*/\bdmn)$};
\node at (1-\figurealpha,1) {$\circ$};
\node [above left] at (1-\figurealpha,1) {$(1-\mathfrak{a}^*,1)$};
\node [below] at (\figurealpha,0) {$\phantom{1}\mathfrak{a}$}; \node at (\figurealpha,0) {$\circ$};
\end{scope}
\end{tikzpicture}
\caption{Values of the parameters $(s,1/p)$ such that the estimate~\eqref{eqn:Besov:estimate} on solutions to the Poisson problem is valid. Because $\mathfrak{a}^*\geq 1-\bdmn$, the illustrated region is convex.}
\label{fig:arbitrary}
\end{figure}
To state our next theorem, we need to recall the definition of  well posedness of the $L^q$-Dirichlet problem. Let $1<q<\infty$; we say that the $L^q$-Dirichlet problem for $L$ is  well posed in $\Omega$ if, for every function $f$ that is Lipschitz continuous and compactly supported, the unique solution $v$ to the Dirichlet problem
\begin{equation}\label{eqn:Dirichlet:cts}
\left\{\begin{gathered}\begin{aligned}
Lv&=0\quad\text{in }\Omega,\\
v&=f\quad\text{on }\partial\Omega,\\
\end{aligned}\\
\int_\Omega |\nabla v|^2\,\delta^{\bdmn+1-\pdmn}<\infty\\
\end{gathered}
\right.\end{equation}
satisfies the nontangential estimate
\begin{equation}
\label{eqn:N:bound}
\|Nv\|_{L^q(\partial\Omega)}\leq C_q\|f\|_{L^q(\partial\Omega)}\end{equation}
where $N$ is the nontangential maximal function given by \begin{equation}\label{eqn:N:intro}
Nv(\xi)=N_{a,\Omega} v(\xi)
=\esssup\{|v(x)|:x\in\Omega,\> |x-\xi|<(1+a)\dist(x,\partial\Omega)\}
\end{equation}
for some positive number~$a$. 
Assuming that our operator and domain are such that the $L^q$-Dirichlet problem is well posed for some $1<q<\infty$ corresponds to adding a point on the $\{s=0\}$-axis in Figure~\ref{fig:arbitrary}. In the case that such a $q$ exists, we provide an extrapolation criterion that extends the range of parameters given in Theorem~\ref{thm:Dirichlet}.
\begin{thm}\label{thm:Dirichlet:Lq}
Suppose that the conditions of Theorem~\ref{thm:Dirichlet} are satisfied.

Suppose furthermore that there is a $q\in (1,\infty)$ such that the $L^q$-Dirichlet problem is well posed. That is, suppose that there exist $q\in (1,\infty)$, $a>0$, and $C_q<\infty$ such that if $f$ is Lipschitz continuous and compactly supported, then the solution $v$ to the problem~\eqref{eqn:Dirichlet:cts} also satisfies the estimate~\eqref{eqn:N:bound}.

Then there is a number $\mathfrak{b}^*\in (0,1]$ such that the conclusions of Theorem~\ref{thm:Dirichlet} are valid for $(s,1/p)$ in the region indicated on the left in Figure~\ref{fig:extrapolation}.

If instead there is a $q^*\in (1,\infty)$ such that the solution $v$ to the problem~\eqref{eqn:Dirichlet:cts}, with $L$ replaced by~$L^*$, satisfies the estimate~\eqref{eqn:N:bound} (with $q$ replaced by~$q^*$), then there is a $\mathfrak{b}\in (0,1]$ such that the conclusions of Theorem~\ref{thm:Dirichlet} are valid for $(s,1/p)$ in the region indicated on the right in Figure~\ref{fig:extrapolation}.

Furthermore, the numbers $\mathfrak{b}^*$ or $\mathfrak{b}$ satisfy the estimates
\begin{equation}\label{eqn:b:lower}
\mathfrak{b}\geq \alpha,
\qquad
\mathfrak{b}^*\geq \alpha^*,
\qquad
\mathfrak{b}\geq 1+\frac{\bdmn}{q^*}-\bdmn,
\qquad
\mathfrak{b}^*\geq 1+\frac{\bdmn}{q}-\bdmn
\end{equation}
where $\alpha$ and $\alpha^*$ are the boundary De Giorgi-Nash exponents for $L$ and~$L^*$.

(If $q$ and $q^*$ both exist, the conclusions of Theorem~\ref{thm:Dirichlet} are valid for $(s,1/p)$ in the region indicated on the bottom in Figure~\ref{fig:extrapolation}.)
\end{thm}

\def\figureq{0.2}
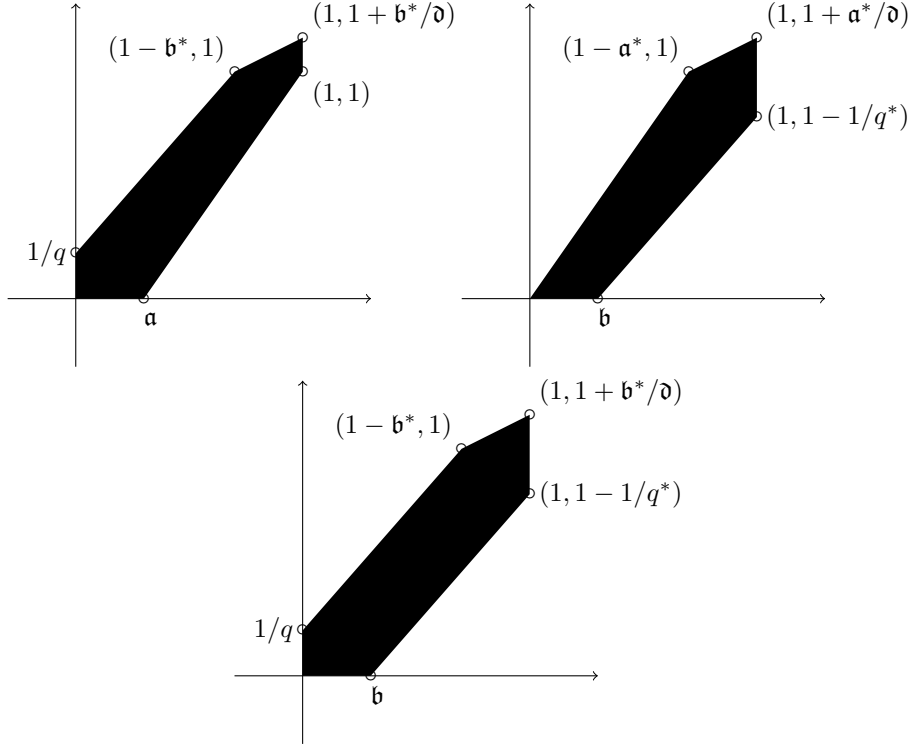
\begin{figure}
\begin{tikzpicture}[scale=3]
\begin{scope}[shift={(0,0)}]
\draw [->] (-0.3,0)--(1.3,0);
\draw [->] (0,-0.3)--(0,1.3);
\node at (1,1) {$\circ$};
\node [below right] at (1,1) {$(1,1)$};
\node at (1,1+\figurealpha/\figuredimen) {$\circ$};
\node [above right] at (1,1+\figurealpha/\figuredimen) {$(1,1+\mathfrak{b}^*/\bdmn)$};
\node at (1-\figurealpha,1) {$\circ$};
\node [above left] at (1-\figurealpha,1) {$(1-\mathfrak{b}^*,1)$};
\node [left] at (0,\figureq) {$1/q$}; \node at (0,\figureq) {$\circ$};
\node [below] at (\figurealpha,0) {$\phantom{1}\mathfrak{a}$}; \node at (\figurealpha,0) {$\circ$};
\fill
  (0,0) 
--(0,\figureq) 
--(1-\figurealpha,1) 
--(1,1+\figurealpha/\figuredimen) 
-- (1,1)
--(\figurealpha,0) 
--cycle;
\end{scope}
\begin{scope}[shift={(2,0)}]
\draw [->] (-0.3,0)--(1.3,0);
\draw [->] (0,-0.3)--(0,1.3);
\node at (1,1-\figureq) {$\circ$};
\node [right] at (1,1-\figureq) {$(1,1-1/q^*)$};
\node at (1,1+\figurealpha/\figuredimen) {$\circ$};
\node [above right] at (1,1+\figurealpha/\figuredimen) {$(1,1+\mathfrak{a}^*/\bdmn)$};
\node at (1-\figurealpha,1) {$\circ$};
\node [above left] at (1-\figurealpha,1) {$(1-\mathfrak{a}^*,1)$};
\node [below] at (\figurealpha,0) {$\phantom{1}\mathfrak{b}$}; \node at (\figurealpha,0) {$\circ$};
\fill
  (1,1) 
--(1,1-\figureq) 
--(\figurealpha,0)
--(0,0)
--(1-\figurealpha,1) 
--(1,1+\figurealpha/\figuredimen) 
--cycle;
\end{scope}
\end{tikzpicture}\\
\begin{tikzpicture}[scale=3]
\draw [->] (-0.3,0)--(1.3,0);
\draw [->] (0,-0.3)--(0,1.3);
\node at (1,1-\figureq) {$\circ$};
\node [right] at (1,1-\figureq) {$(1,1-1/q^*)$};
\node at (1,1+\figurealpha/\figuredimen) {$\circ$};
\node [above right] at (1,1+\figurealpha/\figuredimen) {$(1,1+\mathfrak{b}^*/\bdmn)$};
\node at (1-\figurealpha,1) {$\circ$};
\node [above left] at (1-\figurealpha,1) {$(1-\mathfrak{b}^*,1)$};
\node [left] at (0,\figureq) {$1/q$}; \node at (0,\figureq) {$\circ$};
\node [below] at (\figurealpha,0) {$\phantom{1}\mathfrak{b}$}; \node at (\figurealpha,0) {$\circ$};
\fill
  (1,1)
--(1,1-\figureq) 
--(\figurealpha,0) 
--(0,0)
--(0,\figureq) 
--(1-\figurealpha,1) 
--(1,1+\figurealpha/\figuredimen) 
--cycle;
\end{tikzpicture}
\caption{Values of the parameters $(s,1/p)$ such that the estimate~\eqref{eqn:Besov:estimate} is valid, given solvability of the $L^q$-Dirichlet problem for $L$ (on the left) or the $L^{q^*}$-Dirichlet problem for $L^*$ (on the right) with nontangential estimates.  Because $\mathfrak{b}^*\geq 1-\bdmn+\bdmn/q$, the illustrated region is convex.
In the common case where $q$ and $q^*$ both exist, the estimate~\eqref{eqn:Besov:estimate} is valid for values of $(s,1/p)$ as in the bottom picture.
}

\label{fig:extrapolation}
\end{figure}
In \cite{BarMP26pA}, we proved Theorems \ref{thm:Dirichlet} and \ref{thm:Dirichlet:Lq} in the case $f=0$, that is, for the Poisson problem with homogeneous boundary values. 

Before our results, a wide range of solvability parameters for the Poisson-Dirichlet problem had been obtained only in domains not far from Lipschitz, and under significant restrictions on the coefficients. This is primarily because methods in the existing literature often relied on layer potentials and interpolation between the $L^q$-Dirichlet problem and the regularity problem. In \cite{BarMP26pA}, we used a different approach to overcome this lack of regularity. We did this by first finding a parameter range for which the Poisson problem is well posed under the broad assumptions of Theorem~\ref{thm:Dirichlet}, and then extrapolating (rather than interpolating) between this range and the well posed $L^q$-Dirichlet problem in cases where such well posedness holds. 

In this manuscript, we prove that the general case is obtained as a consequence of the Poisson one. More precisely, we show that we can transform the Dirichlet problem ($f\ne 0$) into a Poisson problem ($f=0$) by using an extension operator $E:\dot B^{s,p}_p(\partial\Omega)\to \dot W^{1,p}_{av,\beta,s}(\Omega)$. Such a use of a bounded extension operator in the study of the Dirichlet problem with boundary data in Besov spaces is not new. The result is known in the half-space since the work of Uspenski\u{\i} (\cite{Usp61}). More recently, in \cite{MitT06}, the authors used bounds for the extension operator on Besov spaces to prove well posedness of the Dirichlet problem for the Hodge Laplacian in Lipschitz domains. See also \cite{Kim07} for the boundedness of extension and trace operators in Besov spaces on Lipschitz domains.
In \cite{MazMS10}, the authors proved extension theorems for Besov spaces with higher smoothness in Lipschitz domains, with applications to higher-order elliptic systems. In \cite{DavFM23B} and \cite{DavFM21}, the authors built an extension operator for Ahlfors regular sets of any dimension (and even for more general sets, with possibly mixed dimensions), but they proved the boundedness of this operator only in $\dot B^{1/2,2}_2$.

The construction of the extension operator is the technical core of this paper and uses many of the techniques of \cite{DavFM23B}.

\subsection{Outline of the paper}
The structure of our paper is as follows. In Section~3, we will define the function spaces we work in, set the assumptions on our domains, and recall some useful preliminary results that we will either use in our proofs or need for the statements of Theorems~\ref{thm:Dirichlet} and~\ref{thm:Dirichlet:Lq} to make sense. In Section~4, we define the extension operator,  and we prove its boundedness in Besov spaces on Ahlfors regular sets. In the proof, we need to distinguish between the cases $p\ge 1$ and $p<1$, in agreement with the different definitions of the Besov spaces in these two cases (see Definition~\ref{dfn:Besov}). In Section~5, we use our results from Section~4, combined with the Poisson version of our main theorems (see \cite{BarMP26pA}), to prove Theorems~\ref{thm:Dirichlet} and~\ref{thm:Dirichlet:Lq}.
\section{Definitions and Preliminaries}
In this section, we define all the domains, function spaces, and quantities appearing in Theorems~\ref{thm:Dirichlet} and~\ref{thm:Dirichlet:Lq}. We also recall some useful lemmas that we proved in \cite{BarMP26pA}. 
\subsection{Geometry and some elliptic estimates}
We work in domains with Ahlfors regular boundaries.
 \begin{defn}\label{dfn:Ahlfors}
A nonempty, possibly unbounded set $\Gamma \subset\R^\dmn$ is said to be $\pbdmn$-Ahlfors regular if there is a constant $A>0$ such that, if $\xi\in\Gamma $ and $r\in\R$ with $0<r\leq \diam \Gamma $, then
\begin{equation*}\frac{1}{A}r^\bdmn\leq \sigma(B(\xi,r)\cap\Gamma )\leq Ar^\bdmn\end{equation*}
where $\sigma$ denotes the $\bdmn$-dimensional Hausdorff measure on subsets of~$\R^\dmn$. 

\end{defn}
If
$\xi\in \Gamma$ and $r>0$, then we will denote
\begin{equation*}\Delta(\xi,r)=B(\xi,r)\cap\Gamma
.\end{equation*}
Ahlfors regularity allows us to estimate the size of regions near the boundary. This is the content of the following lemma, which was proven in full generality in \cite{BarMP26pA}, but some important precursors may be found in \cite[Lemma 2.3]{DavFM21}.

From now on, for any two nonnegative real numbers $A$ and $B$, we denote $A\approx B$ to mean that there exist positive constants $C$ and $c$, depending on parameters to be specified case by case, such that $cB\le A\le CB$.

\begin{lem}[]{\cite[Lemma 2.6]{BarMP26pA}}\label{lem:close:to:Ahlfors}
 Let $\Gamma\subset\R^\dmn$ be $\bdmn$-Ahlfors regular for some $0<\bdmn<\dmn$ with Ahlfors constant~$A$. If $\xi\in\Gamma$ and $0<r\leq R<2\diam\Gamma$, then
\begin{equation}
\label{eqn:close:to:Ahlfors} 
|\{y\in B(\xi,R):\dist(y,\Gamma)<r\}|
\approx R^\bdmn r^{\dmn-\pbdmn}
\end{equation}
where the implicit constants depend only on~$A$, $\bdmn$ and~$\dmn$.

Furthermore, if $\xi\in \Gamma$, $0<R<\infty$, and $\varpi>0$, then
\begin{equation}\label{eqn:delta:in:ball}
\int_{B(\xi,R)} \dist(y,\Gamma)^{\pbdmn-\pdmn+\varpi}\,dy
\approx R^{\pbdmn+\varpi}
\end{equation}
where the implicit constants depend only on~$A$, $\bdmn$, $\dmn$, and~$\varpi$.
\end{lem}

Our domains will satisfy certain quantitative assumptions on openness and connectedness (as required by Theorem~\ref{thm:Dirichlet}).

If $d=n-1$, we require domains to satisfy the weak local John condition, which we now define.

\begin{defn}\label{dfn:local:John}
Let $\Omega\subset\R^\dmn$ be open.

Let $x\in\partial\Omega$ and let $y\in\Omega$. We say that $\zeta$ is a $\lambda$-carrot path from $x$ to $y$ if $\zeta$ is a connected rectifiable path, which we require to be parameterized by arc length, such that $\zeta:[0, T]\to\Omega\cup\{x\}$ is continuous, $\zeta(0)=x$, $\zeta(T)=y$, and $\delta(\zeta(t))=\dist(\zeta(t),\partial\Omega)\geq \lambda t$ for all $t\in [0, T]$.

Suppose that $\partial\Omega$ is $\bdmn$-Ahlfors regular for some $0<\bdmn<\dmn$.
Let $y\in\Omega$. Let $\theta$, $\lambda\in (0,1]$ and let $N\geq 2$. We say that $y\in\Omega$ is a $(\theta,\lambda,N)$-weak local John point if there is a Borel set $F\subset B(y,N\delta(y))\cap\partial\Omega$ such that $\sigma(F)\geq \theta\sigma( B(y,N\delta(y))\cap\partial\Omega )$ and such that, if $x\in F$, then there is a $\lambda$-carrot path from $x$ to~$y$.

We say that $\Omega$ is a weak local John domain if there exist constants $\theta\in (0,1]$, $\lambda\in (0,1]$ and $N\geq 2$ such that every $y\in\Omega$ is a $(\theta,\lambda,N)$-weak local John point.
\end{defn}

When the dimension $\bdmn$ of $\partial\Omega$ is strictly larger than $\dmnMinusOne$, we will require the stronger conditions of interior corkscrews and Harnack chains. 

\begin{defn}[Interior corkscrew condition]\label{dfn:iCS}
We say that $\Omega$ satisfies the {interior corkscrew condition} if there is a constant $\corkscrew\in (0,1)$ such that, if $x\in \partial\Omega$ and $0<r<\diam\Omega$, then there is a point $A_r(x)$ (called a corkscrew point) such that
\begin{equation*}
A_r(x)\in B(x,r)\quad\text{and}\quad B(A_r(x), \corkscrew r)\subset\Omega.
\end{equation*}
\end{defn}

\begin{defn}[Harnack chains]\label{dfn:Harnack}
Let $\Omega\subset\R^\dmn$. We say that $\Omega$ satisfies the \emph{Harnack chain condition} if there are constants $M\geq 1$ and $m\in (0,1)$ such that if $x$, $y\in \Omega$, then there is an integer~$N$, depending only on the ratio $|x-y|/\delta$ for $\delta=\min(\dist(x,\partial\Omega),\dist(y,\partial\Omega))$, such that there is a chain of open balls $\{B_k\}_{k=1}^N$ with $x\in B_1$, $y\in B_N$, $B_k\cap B_{k+1}\neq\emptyset$ for all $1\leq k<N$, and $m\diam(B_k)\leq \dist(B_k,\partial\Omega)\leq M\diam(B_k)$.
\end{defn}

\begin{rmk}
By connecting corkscrew points with Harnack chains, we can construct carrot paths to a point $y$ in the domain from \emph{any} nearby boundary point~$x$. Thus, any domain that satisfies both the interior corkscrew and Harnack chain conditions is a weak local John domain.

If $0<\bdmn<\dmnMinusOne$, then the only possible choice for~$\Omega$ is $\Omega=\R^\dmn\setminus\partial\Omega$ (that is, a lower dimensional boundary cannot divide $\R^\dmn$ into multiple open connected components).
Furthermore, if $\bdmn<\dmnMinusOne$ and $\Gamma$ is unbounded (as required by Theorem~\ref{thm:Dirichlet}), then, by \cite[Lemmas 2.1 and~11.6]{DavFM21}, $\Omega$ satisfies both the interior corkscrew and Harnack chain conditions, and so $\Omega$ is necessarily a weak local John domain.

If $\bdmn\geq\dmnMinusOne$, then the weak local John condition is not guaranteed by the Ahlfors condition. However, it is often a necessary condition in the theory; see, for example, \cite{AzzHMMT20, HofMMTZ21}, where it is shown that given Ahlfors regularity, interior corkscrews, and certain natural conditions on the coefficients~$A$, the weak local John condition is a necessary condition for well posedness of the $L^q$ Dirichlet problem (that is, for the estimate~\eqref{eqn:N:bound}).

In the companion paper \cite{BarMP26pA}, we conjectured that the interior corkscrew and Harnack chain conditions in Theorem~\ref{thm:Dirichlet} can be relaxed to the weak local John condition even in the lower codimensional case $\dmnMinusOne<\bdmn<\dmn$. Indeed, we used this additional assumption only because existing literature needs it to ensure that certain elliptic estimates hold up to the boundary (e.g., Lemma \ref{lem:boundary:DGN} below).
\end{rmk}

We now define the elliptic operators we work with, and we state two results on elliptic equations that will give meaning to some quantities in Theorems \ref{thm:Dirichlet} and \ref{thm:Dirichlet:Lq}.

We work with an elliptic operator in divergence form
\begin{equation}\label{eq:operator}
L:=-\Div(A\nabla),
\end{equation}
where $A$ satisfies the positive definiteness and boundedness conditions.
\begin{defn}
Let $\Omega\subset\R^\dmn$ be an open set with $\bdmn$-Ahlfors regular boundary, and let $A:\Omega\to M_\dmn(\R)$ be a measurable function with values in the space of $\dmn\times\dmn$-matrices with real coefficients. We say that $A$ satisfies the positive definiteness and boundedness conditions if there are positive numbers $\lambda$ and $\Lambda$ such that
\begin{align}
\label{eqn:elliptic:introduction}
{\xi}\cdot A(x)\xi
&\geq
\delta(x)^{\bdmn+1-\pdmn} \lambda\|\xi\|^2 
,\\ 
\label{eqn:elliptic:introduction:1}
|\eta\cdot A(x)\xi|
&\leq \delta(x)^{\bdmn+1-\pdmn} \Lambda\|\xi\|\,\|\eta\| \end{align}
for all $\eta$, $\xi\in\R^\dmn$ and all $x\in\Omega$, where
\begin{equation*}\delta(x):=\dist(x,\partial\Omega).\end{equation*}

\end{defn}

We say that $Lu=-\Div \vec H$ in an open set $\Psi\subseteq\Omega$ if
$\vec H\in L^1_{loc}(\Psi)$, $u\in W^{1,1}_{loc}(\Psi)$, and
\begin{equation}\label{eqn:L}
\int_\Psi\nabla\varphi\cdot A\nabla u=\int_\Psi \nabla\varphi\cdot \vec H\text{ for all }\varphi\in C^\infty_0(\Psi).
\end{equation}
We want to prescribe boundary values in the averaged sense.
\begin{defn}\label{dfn:trace}
Let $\Omega\subset\R^\dmn$ be open, let $u\in L^1_{loc}(\overline\Omega)$, and let $\xi\in \partial\Omega$.

If $U\in\R$, we say that
$\Tr u(\xi)=U$ if
\begin{equation*}\lim_{r\to 0^+} \frac{1}{r^\dmn} \int_{B(\xi,r)\cap\Omega} |u-U|=0.\end{equation*}
\end{defn}

The gradient of solutions has the same local integrability as $\vec {H} $ in a neighbourhood of $2$.
This is a classic result of Meyers \cite{Mey63}, stating the following. There is a $p^+=p^+_L>2$ such that, if $2\leq\beta<p^+$, then there is a $C_{\beta}<\infty$ such that if $\vec H\in L^\beta(B(y,2r))$, $B(y,4r)\subset\Omega$, and $Lu=-\Div(\delta^{1+\bdmn-\pdmn}\vec H)$ in $B(y,2r)$, then
\begin{align}
\label{eqn:interior:Meyers}
\biggl(\fint_{B(y,r)}|\nabla u|^\beta \biggr)^{1/\beta}
&\leq C_{q,\beta}\biggl(\fint_{B(y,2r)} |\nabla u|^2\biggr)^{1/2}
+C_{q,\beta}\biggl(\fint_{B(y,2r)}|\vec H|^\beta \biggr)^{1/\beta}.
\end{align}
It is possible to bound $p^+-2$ from below by a positive number depending only on the ellipticity constants of~$A$. In many cases, it is possible to improve beyond these estimates on~$p^+_L$; for example, if $A$ is constant\footnote{By the ellipticity conditions~\eqref{eqn:elliptic:introduction},~\eqref{eqn:elliptic:introduction:1}, this can only happen if $\bdmn=\dmnMinusOne$.} then $p^+_L=\infty$.
Meyers' estimate~\eqref{eqn:interior:Meyers} allows us to extend the subaveraging from $L^2$ to $L^\beta$ in the bound~\eqref{eqn:Besov:estimate} (see \cite{BarMP26pA}).

The standard De Giorgi-Nash-Moser theory ensures that solutions with $\vec {H} = 0$ are locally Hölder continuous inside $\Omega$, and in the literature there are many generalizations of the De Giorgi-Nash-Moser Hölder continuity result that hold up to the boundary. In the domains we consider, the same holds up to the boundary for solutions with zero trace. 
\begin{lem}[Boundary Hölder continuity of solutions]\label{lem:boundary:DGN}
Let $\Omega\subset\R^\dmn$ be as in Theorem~\ref{thm:Dirichlet}. 

There exists a constant $\alpha>0$ depending only on the standard parameters with the following significance.

Let $\xi\in\partial\Omega$ and let $r\in(0,\diam\partial\Omega)$. Let $u$ be such that
\begin{align*}&\int_{B(\xi,2r)\cap\Omega} |\nabla u|^2 \delta^{-\pdmn+\bdmn+1}<\infty,
\\
Lu&=0 \text{ in }B(\xi,2r)\cap\Omega,
\\
\Tr u&=0 \text{ on }\Delta(\xi,2r).\end{align*}
Then for all $s$ with $0<s\leq r/2$ we have that
\begin{equation*}
\sup_{B(\xi,s)\cap\Omega} |u|\leq C\frac{s^\alpha}{r^\alpha} \frac{1}{r^{1+\bdmn} } \int_{B(\xi,r)\cap\Omega} |u| \delta^{1+\bdmn-\pdmn}
\approx
\frac{s^\alpha}{r^\alpha} \frac{1}{\int_{B(\xi,r)} \delta^{1+\bdmn-\pdmn}} \int_{B(\xi,r)\cap\Omega} |u| \delta^{1+\bdmn-\pdmn}
.\end{equation*}
\end{lem}
The proof in the case $\bdmn=1$ and $\dmn=2$ in a weak local John domain is \cite[Section 6.2]{BarMP26pA}; citations for the result in all other cases may be found in \cite[Lemma 6.2]{BarMP26pA}.
\subsection{Besov spaces on Ahlfors regular sets}\label{sec:bdry:space:dfn}

We now introduce the spaces in which we want to prescribe boundary values.

If $\Gamma$ is a $\bdmn$-Ahlfors regular set in $\R^\dmn$, and if $d$ denotes Euclidean distance and $\sigma$ denotes surface measure, then $(\Gamma,d,\sigma)$ is a space of homogeneous type. Homogeneous Besov (and Triebel-Lizorkin) spaces on spaces of homogeneous type were introduced in \cite{HanS94} and defined there in terms of approximations to the identity. Their definition has been generalized in other papers, for example, in \cite{Han98, HanMY08}. This definition is very general but also very technical. We will instead use two related characterizations of Besov spaces.

It was shown in \cite[Theorem~4.1]{MulY09} that the standard definition of the Besov space $\dot B^{s,p}_p(\Gamma)$ from \cite{HanMY08} on an Ahlfors regular set (or on certain more general sets) coincides with the following definition if $p\geq 1$ and $0<s<1$.

\begin{defn}
\label{dfn:Slobodeckij}
Let $\Gamma\subseteq\R^\dmn$ be a $\pbdmn$-Ahlfors regular set, and let $-\frac{\bdmn}{p}<s<\infty$. The Slobodeckij space $\dot \Lambda^{s,p}(\Gamma)$ is defined to be the set of all equivalence classes of locally integrable functions $f:\Gamma\to\R$ such that the norm (or quasi-norm, if $0<p<1$)
\begin{equation*}\|f\|_{\dot \Lambda^{s,p}(\Gamma)}
=
\biggl(\int_{\Gamma}\int_\Gamma 
\frac{|f(x)-f(y)|^p}{|x-y|^{\bdmn+ps}}
\,d\sigma(y)\,d\sigma(x)\biggr)^{1/p}
\end{equation*}
is finite, where the equivalence relation is the addition of constants.
\end{defn}
\begin{rmk}
It is clear that if $1\leq p\leq\infty$, then the Slobodeckij spaces are Banach spaces. 
We think of $s$ as a parameter measuring smoothness, while $p$ measures integrability.
We will primarily use the Slobodeckij spaces in the case $p\geq 1$ (in fact, we will use them as our definition of Besov spaces). We state our main results exclusively for the range $0<s<1$; outside this interval, the Slobodeckij seminorm fails to capture fractional smoothness and leads to degenerate spaces.
\end{rmk}

Recall that if $p\leq 1$, the Lebesgue spaces $L^p(\R^\bdmn)$ are in many ways unsuited to harmonic analysis and the Hardy spaces $H^p(\R^\bdmn)$ are often studied instead. Furthermore, recall that Hardy spaces have an atomic characterization. Atomic characterizations also exist for Besov spaces; see \cite{FraJ85} in the case of Euclidean spaces and \cite{HanLY99, HanY03} in the case of the inhomogeneous Besov space $B^{s,p}_p(\Gamma)=\dot B^{s,p}_p(\Gamma)\cap L^p(\Gamma)$ on a space $\Gamma$ of homogeneous type. 
We prefer to work with homogeneous Besov spaces, and so our atomic characterization differs somewhat from that of \cite{HanY03}. It is natural to conjecture that our atomic spaces $\dot A^{s,p}(\Gamma)$ (see Definition~\ref{dfn:atom:space}) coincide with the homogeneous Besov spaces $\dot B^{s,p}_p(\Gamma)$ as defined, for example, in \cite{Han98}, but we have not been able to find a proof of this result in the literature.

Following \cite{HanLY99, HanY03}, we define blocks as follows.
\begin{defn}[{\cite[Definition~2.1]{HanLY99}; \cite[Definition~7]{HanY03}}]
\label{dfn:atom}
Let $\Gamma \subset\R^\dmn$ be a $\pbdmn$-Ahlfors regular set, let $z\in\Gamma$, and let $\varrho>0$. 
A measurable function $a:\Gamma \to\R$ is said to be a 
\emph{block in $\Delta(z,\varrho)$} if 
\begin{itemize}
\item $\supp a\subset \Delta(z,\varrho)$,
\item $|a(x)|\leq \varrho^{-\pbdmn/2}$ for all $x\in \Delta(z,\varrho)$ (and thus for all $x\in\Gamma$),
\item $|a(x)-a(y)|\leq |x-y|\varrho^{-\pbdmn/2-1}$ for all $x$, $y\in\Gamma$.
\end{itemize}
\end{defn}

We have the following embedding results for blocks in the spaces defined above.
\begin{lem}\label{lem:atom:in:M1p} If $a$ is a 
block in $\Delta(z,\varrho)$, if $0<s<1$, and if $0<p\leq\infty$, then there is a constant $C_{A,p,s}$ depending only on $A$, $p$, and~$s$ such that
\begin{align*}
\|a\|_{L^p(\Gamma)}&\leq A^{1/p}\varrho^{\pbdmn/p-\pbdmn/2}
,\\
\|a\|_{\dot \Lambda^{s,p}(\Gamma)}
&\leq C_{A,s,p}\varrho^{\pbdmn/p-\pbdmn/2-s}
,\end{align*}
where $A$ is the Ahlfors constant for~$\Gamma$.
\end{lem}
\begin{proof}
Fix some such block~$a$. 

The bound on~$\|a\|_{L^p(\Gamma)}$ follows immediately from the bound on $\|a\|_{L^\infty(\Gamma)}$, the support condition on~$a$, and the Ahlfors condition on $\sigma(\Delta(z,\varrho))$.

We turn to $\|a\|_{\dot \Lambda^{s,p}(\Gamma)}$. If $p=\infty$ the result follows immediately by considering the two cases $|x-y|<\varrho$ and $|x-y|\geq \varrho$ separately (and we may take $C_{A,s,p}=2$). Otherwise, if $x\in\Delta(z,2\varrho)$ then $|a(x)|\leq \varrho^{-\pbdmn/2}$ and $|a(x)-a(y)| \leq |x-y|\varrho^{-\pbdmn/2-1}$ for all $y\in \Delta(z,2\varrho)$. Thus, for any $x\in \Delta(z,2\varrho)$,
\begin{align*}
\int_{\Delta(z,2\varrho)}
\frac{|a(x)-a(y)|^p}{|x-y|^{\bdmn+ps}}
\,d\sigma(y)
&\leq
\sum_{j=-2}^\infty 
\int_{\Delta(x,2^{-j}\varrho)\setminus \Delta(x,2^{-j-1}\varrho)}
\frac{|x-y|^p\varrho^{-p\pbdmn/2-p}}{|x-y|^{\bdmn+ps}}
\,d\sigma(y)
,\\
\int_{\Gamma\setminus\Delta(z,2\varrho)}
\frac{|a(x)-a(y)|^p}{|x-y|^{\bdmn+ps}}
\,d\sigma(y)
&\leq
\sum_{j=0}^\infty 
\int_{\Delta(x,2^{j+1}\varrho)\setminus \Delta(x,2^{j}\varrho)}
\frac{\varrho^{-p\pbdmn/2}}{|x-y|^{\bdmn+ps}}
\,d\sigma(y)
.\end{align*}
Applying the Ahlfors regularity of~$\Gamma$, we see that
\begin{align*}
\int_{\Delta(z,2\varrho)}
\int_{\Gamma}
\frac{|a(x)-a(y)|^p}{|x-y|^{\bdmn+ps}}
\,d\sigma(y)\,d\sigma(x)
&\leq
C_{A,p,s}\varrho^{\bdmn-p\pbdmn/2-ps}
.\end{align*}
The same bound is of course valid with the roles of $x$ and~$y$ reversed, while if $x$ and $y$ are both in $\Gamma\setminus \Delta(z,2\varrho)$ then the integrand is zero; this establishes the bound on $\|a\|_{\dot \Lambda^{s,p}(\Gamma)}$.
\end{proof}

As an immediate consequence of this lemma and of the completeness of the space $\dot \Lambda^{s,p}(\Gamma)$, 
we have that if $\{a_k\}_{k=1}^\infty$ is a sequence of functions, each $a_k$ is a block in $\Delta(z_k,\varrho_k)$ for some $z_k\in\Gamma$ and some $\varrho_k>0$, and if $\{\lambda_k\}_{k=1}^\infty$ is a sequence of real numbers such that
\begin{equation*}\sum_{k=1}^\infty |\lambda_k| \varrho_k^{\bdmn/p-\pbdmn /2-s}<\infty,\end{equation*}
then $\sum_{k=1}^\infty \lambda_k a_k$ is a convergent series in $\dot \Lambda^{s,p}(\Gamma)$ provided $0<s<1$, $1\leq p\leq \infty$. 

If $p\leq1$, then 
by the standard embedding results for sequence spaces,
\begin{equation*}
\sum_{k=1}^\infty |\lambda_k| 
\varrho_k^{\bdmn -\pbdmn /2-(s+\bdmn-
\bpdmn/p)}
\leq
\biggl(\sum_{k=1}^\infty |\lambda_k|^p \varrho_k^{\bdmn -\pbdmn p/2-sp}\biggr)^{1/p}\end{equation*}
and so if the right-hand side is finite then $\sum_{k}\lambda_k a_k$ converges in $\dot \Lambda^{s+\bdmn-
\bpdmn/p,1}(\Gamma)$. We may thus define atomic spaces for $\bdmn/(\bdmn+s)<p\leq 1$ as follows.
\begin{defn}
\label{dfn:atom:space}
Let $\Gamma \subset\R^\dmn$ be a $\pbdmn $-Ahlfors regular set, 
let $0<s< 1$, and let $\pbdmn/(\bdmn+s)<p\leq 1$.

We let 
\begin{multline*}\dot A^{s,p}(\Gamma)=\Bigl\{\sum_{k=1}^\infty \lambda_k a_k: \sum_{k=1}^\infty |\lambda_k|^p \varrho_k^{\bdmn-\pbdmn p/2-sp}<\infty\\
\text{where each $a_k$ is a block in $\Delta(z_k,\varrho_k)$ for some $z_k\in\Gamma$}\Bigr\}
\end{multline*}
where the convergence is in the Banach space $\dot \Lambda^{s+\bdmn-
\bpdmn/p,1}(\Gamma)$,
and let the norm in $\dot A^{s,p}(\Gamma)$ be given by
\begin{multline*}\|f\|_{\dot A^{s,p}(\Gamma)}=\inf\Bigl\{\Bigl(\sum_{k=1}^\infty |\lambda_k|^p \varrho_k^{\bdmn -\pbdmn p/2-sp}\Bigr)^{1/p}:
\\f=\sum_{k=1}^\infty \lambda_k a_k\text{ for some 
blocks $a_k$ in $\Delta(z_k,\varrho_k)$}\Bigr\}.\end{multline*}
\end{defn}

\begin{rmk}\label{rmk:atom:block}
A block $a$ in $\Delta(z,\varrho)$ is called an atom if, in addition to the conditions of being a block, it satisfies $\int_\Gamma a\,d\sigma=0$.
Let $\widetilde A^{s,p}(\Gamma)$ be defined analogously to $\dot A^{s,p}(\Gamma)$ with the functions $a_k$ required to be atoms in $\Delta(z_k,\varrho_k)$ (and not merely to be blocks).

It is often customary to consider Hardy-type spaces based on atoms. However, if $\Gamma$ is unbounded, $\widetilde A^{s,p}(\Gamma)=\dot A^{s,p}(\Gamma)$ with an equivalence of norms, and therefore this restriction is not meaningful in the present context.  

Indeed, let $\varphi:\R^\dmn\to[0,\infty)$ be smooth and satisfy $\varphi(x)=0$ if $|x|\geq 1$, $|\nabla\varphi|\leq 1$, and $\varphi(x)=1/3$ if $|x|<1/3$. If $z\in\Gamma$ and $r>0$, then
\begin{equation*}\varphi_{z,r}(x)= r^{-\pbdmn/2}\varphi((x-z)/r)\end{equation*}
is a block in $\Delta(z,r)$. Furthermore,
\begin{equation}\label{eqn:lbatom}
\int_{\Gamma} \varphi_{z,r}\,d\sigma
\geq \frac{1}{3A} (r/3)^{\bdmn-\bdmn/2}=cr^{\bdmn/2}.\end{equation}
If $b$ is a block in $\Delta(z,\varrho)$, then $\int b\,d\sigma\leq A\varrho^{\bdmn/2}$. We have that
\begin{equation*}b(x)=b(x)-\frac{\int_{\Gamma} b\,d\sigma }{\int_\Gamma \varphi_{z,\varrho}\,d\sigma} \varphi_{z,\varrho}(x)
+\sum_{j=1}^\infty \frac{\int_{\Gamma} b\,d\sigma }{\int_\Gamma \varphi_{z,2^{j-1}\varrho}\,d\sigma} \varphi_{z,2^{j-1}\varrho}(x)-\frac{\int_{\Gamma} b\,d\sigma }{\int_\Gamma \varphi_{z,2^{j}\varrho}\,d\sigma} \varphi_{z,2^{j}\varrho}(x).\end{equation*}
But 
\begin{equation*}b-\frac{\int_{\Gamma} b\,d\sigma }{\int_\Gamma \varphi_{z,\varrho}\,d\sigma} \varphi_{z,\varrho}=\biggl(1+ \frac{|\int_{\Gamma} b\,d\sigma|} {|\int_\Gamma \varphi_{z,\varrho}\,d\sigma|} \biggr)a_0\end{equation*} 
for a block $a_0$ in $\Delta(z,\varrho)$; but $a_0$ integrates to zero and so must be an atom. Note that the quantity $1+ \frac{|\int_{\Gamma} b\,d\sigma|} {|\int_\Gamma \varphi_{z,\varrho}\,d\sigma|}$ is bounded above depending only on $A$ and~$\bdmn$. Similarly
\begin{align*}\frac{\int_{\Gamma} b\,d\sigma }{\int_\Gamma \varphi_{z,2^{j-1}\varrho}\,d\sigma} \varphi_{z,2^{j-1}\varrho}-\frac{\int_{\Gamma} b\,d\sigma }{\int_\Gamma \varphi_{z,2^{j}\varrho}\,d\sigma} \varphi_{z,2^{j}\varrho}
&=\biggl(2^{\bdmn/2}\frac{|\int_{\Gamma} b\,d\sigma|} {|\int_\Gamma \varphi_{z,2^{j-1}\varrho}\,d\sigma|} 
+\frac{|\int_{\Gamma} b\,d\sigma|} {|\int_\Gamma \varphi_{z,2^{j}\varrho}\,d\sigma|} \biggr)a_j\\
&=\lambda_ja_j
\end{align*}
for an atom $a_j$ in $\Delta(z,2^j\varrho)$. Further $|\lambda_j|\leq C2^{-j\bdmn/2}$.

Thus
\begin{equation*}\sum_{j=0}^\infty \lambda_j^p (2^j\varrho)^{\bdmn-\bdmn p/2-ps}
\leq C\varrho^{\bdmn-\bdmn p/2-ps}\sum_{j=0}^\infty 2^{j(\bdmn-\bdmn p-ps)}\end{equation*}
and the sum converges whenever $p>\frac{\bdmn}{\bdmn+s}$, and so every block is in the analogue to $\dot A^{s,p}(\Gamma)$ with atoms required.

If $\Gamma$ is bounded, then~\eqref{eqn:lbatom} holds only for $r>0$ such that $\frac{r}{3}\le\diam(\Gamma)$. Therefore, we cannot write a block $b$ as a telescopic sum of atoms as we did before. Instead, we can write
\begin{align*}
b(x)&=\frac{\int_{\Gamma} b\,d\sigma }{\int_\Gamma \varphi_{z,2^{J}\varrho}\,d\sigma} \varphi_{z,2^{J}\varrho}(x)+b(x)-\frac{\int_{\Gamma} b\,d\sigma }{\int_\Gamma \varphi_{z,\varrho}\,d\sigma} \varphi_{z,\varrho}(x)\\
&+\sum_{j=1}^{J} \frac{\int_{\Gamma} b\,d\sigma }{\int_\Gamma \varphi_{z,2^{j-1}\varrho}\,d\sigma} \varphi_{z,2^{j-1}\varrho}(x)-\frac{\int_{\Gamma} b\,d\sigma }{\int_\Gamma \varphi_{z,2^{j}\varrho}\,d\sigma} \varphi_{z,2^{j}\varrho}(x),
\end{align*}
where $\rho\le\diam(\Gamma)$, $b$ is a block in $\Delta(z,\rho)$, and $J$ is the first positive integer such that $\diam(\Gamma)\le \frac{2^J\rho}{3}< 2\diam(\Gamma)$. Because of the choice of $J$, $\frac{\int_{\Gamma} b\,d\sigma }{\int_\Gamma \varphi_{z,2^{J}\varrho}\,d\sigma} \varphi_{z,2^{J}\varrho}(x)=\frac{\int_{\Gamma} b\,d\sigma }{\int_\Gamma \varphi_{z,2^{J}\varrho}\,d\sigma}$ is constant. Therefore, in the case of $\Gamma$ bounded we conclude that $\widetilde A^{s,p}(\Gamma)=\dot A^{s,p}(\Gamma)$ with equivalent norms modulo adding constants.

\end{rmk}

The Besov spaces used in Theorems~\ref{thm:Dirichlet} and~\ref{thm:Dirichlet:Lq}    are defined as follows.
\begin{defn}\label{dfn:Besov}Let $\Gamma\subset\R^\dmn$ be a $\bdmn$-Ahlfors regular set, let $0<s<1$, and let $\bdmn/(\bdmn+s)<p\leq \infty$.
We then define
\begin{equation*}
\dot B^{s,p}_p(\Gamma)=\begin{cases} \dot A^{s,p}(\Gamma), & p<1,\\ \dot \Lambda^{s,p}(\Gamma), & p\geq 1.\end{cases}
\end{equation*}
\end{defn}

\begin{rmk}
The theory of  Besov spaces on Euclidean space $\dot B^{s,p}_q(\R^\dmn)$ defines and considers these spaces for all $s\in\R$ and all $p$, $q>0$. However, in the theory of Besov spaces on spaces of homogeneous type $\dot B^{s,p}_q(\Gamma)$, the theory is generally only formulated in the case where the smoothness parameter $s$ lies in a range depending on the structure of the space~$\Gamma$ (but always in $[-1,1]$), and where $p$, $q$ are also bounded from below. See, as recent general formulations, \cite{Han98, HanMY08}.

The paper \cite{FraJ85} provides an atomic decomposition of $\dot B^{s,p}_q(\R^\dmn)$ for any~$s\in\R$. The atoms used must satisfy conditions depending on the values of $s$ (and $p$ and~$q$); there are no functions that may be used as atoms for all possible~$s$. It is thus common in the Euclidean theory to speak of $\dot {B} ^ {s,p}_q$-atoms rather than simply atoms. 

However, if $s$ is constrained to lie in $(-1,1)$ and $p$, $q$ satisfy appropriate lower bounds, then a single set of atoms may be used for all such $\dot B^{s,p}_q(\Gamma)$; see \cite{HanLY99, HanY03}. Thus, in the theory of Besov spaces on spaces of homogeneous type, ``atoms'' denotes what in the Euclidean theory would be referred to as $\dot B^{s,p}_q$-atoms for all $s\in (-1,1)$ and $p$, $q>\bdmn/(\bdmn+1+s)$. The word ``block'' denotes what in the Euclidean theory would be referred to as $\dot B^{s,p}_q$-atoms for all $s\in (0,1)$ and $p$, $q>\bdmn/(\bdmn+1+s)$ (but not necessarily for $s\leq 0$). We have chosen to follow \cite{HanLY99, HanY03} and refer to the functions in Definition~\ref {dfn:atom} as blocks rather than atoms.

(This is also the reason that blocks are normalized such that their $L^2$ norm, rather than their $\dot \Lambda^{s,p}$ norm, is bounded independently of~$\varrho$.)
\end{rmk}

\subsection{Whitney decompositions of open sets}\label{sec:dfn:Whitney}

We define Whitney decompositions of open sets as follows.

\begin{defn}\label{dfn:Whitney}
A dyadic cube in $\R^\dmn$ is a set of the form \begin{equation*}Q=[j_12^{k},(j_1+1)2^k)\times\dots\times[j_\dmn2^k,(j_\dmn+1)2^k)\end{equation*} for some integers $k$, $j_1,\dots,j_\dmn$. The side length of such a cube is~$\ell(Q)=2^k$. Observe that the set of all cubes of side length $2^k$ forms a partition of~$\R^\dmn$.

Suppose that $Q$ is a dyadic cube in $\R^\dmn$. 
We let $P(Q)$ denote its dyadic parent, that is, the unique dyadic cube that satisfies both $Q\subset P(Q)$ and $\ell(P(Q))=2\ell(Q)$. 

If $\Gamma\subset\R^\dmn$ is a nonempty proper closed subset of~$\R^\dmn$, then the Whitney decomposition $\mathcal{G}=\mathcal{G}(\Gamma)$ of $\R^\dmn\setminus\Gamma$ is given by
\begin{equation*}\mathcal{G}(\Gamma)= \{Q\text{ dyadic}: \dist(Q,\Gamma)>\diam Q,\> \dist(P(Q),\Gamma) \leq  \diam P(Q)\}.\end{equation*}

\end{defn}

We will frequently need the closest point to a cube; thus, if $Q\in \mathcal{G}$, we let $\xi_Q$ satisfy
\begin{equation}\label{eqn:cube:point}
\xi_Q\in\Gamma,\qquad \dist(\xi_Q,Q)=\dist(\Gamma,Q).
\end{equation}
Because $\Gamma$ is closed, $\xi_Q$ must exist.
For any Whitney cube, we fix such a point, and we define the projection of the cube on the boundary as
\begin{equation}\label{dfn:Delta:sub:Q}
\Delta_Q:=\Delta(\xi_Q,\diam Q)=B(\xi_Q,\diam Q)\cap\Gamma.
\end{equation}

\subsection{Basic properties of Whitney decompositions}
Whitney decompositions are a standard tool in harmonic analysis and have many well known properties; for convenience, we will recall these properties as established in \cite{BarMP26pA}.

Observe that if $x\in\R^\dmn\setminus\Gamma$ then $x\in Q$ for exactly one $Q\in\mathcal{G}(\Gamma)$.
If $Q\in \mathcal{G}$, then 
\begin{equation*}\dist(Q,\Gamma)-\diam Q \leq\dist(P(Q),\Gamma)\leq \diam(P(Q))=2\diam Q ,\end{equation*}
and so 
\begin{equation}\label{eqn:Whitney}
\dist(Q,\Gamma)\leq 3\diam Q . 
\end{equation}
This, together with Definition~\ref{dfn:Whitney} of Whitney cubes, implies
\begin{equation*}
\dist(Q,\Gamma)\approx\diam Q.
\end{equation*}
In particular,
\begin{equation}
\label{eqn:R:B:size}
\text{if }Q\in\mathcal{G},\text{ } x\in\Gamma,\text{ and }Q\cap B(x,r)\neq\emptyset\text{, then }\ell(Q)\leq Cr
.\end{equation}
Indeed, $\ell(Q)\approx\diam Q\approx \dist(Q,\Gamma)\leq r$.

We now recall the following lemmas on useful properties of Whitney cubes (see \cite{BarMP26pA} for proofs). If $s>0$ and $Q$ is a cube, we denote by $sQ$ the cube having the same center as $Q$ and with side length $\ell(sQ)=s\ell(Q)$.
\begin{lem}\label{lem:dilate:Whitney}
Let $\Gamma\subset\R^\dmn$ be a nonempty proper closed set, and let $\mathcal{G}=\mathcal{G}(\Gamma)$  be as in Definition~\ref{dfn:Whitney}.

If $Q\in\mathcal{G}$, then
\begin{align}
\label{eqn:Whitney:dilate:distance:min}
\dist(2Q,\Gamma)&\geq \frac{1}{2}\diam Q,
\\ 
\label{eqn:Whitney:dilate:distance:max}
\sup_{x\in 2Q} |x-\xi_Q|&\leq\frac{9}{2} \diam Q.\end{align}
If $Q$, $R\in\mathcal{G}$, then
\begin{equation}
\label{eqn:Whitney:dilate:distance}
\frac{1}{2}\diam Q
-\frac{9}{2}\diam R
\leq\dist(2Q,2R)\end{equation}
and so if $2Q\cap 2R\neq \emptyset$, then
\begin{equation*} \dist(Q,\Gamma)\approx  \diam Q\approx \diam R \approx \dist(R,\Gamma)\end{equation*}
with the implicit constants depending only on~$\dmn$.
In particular, there are at most $C_{\dmn}$ such cubes~$R$.

\end{lem}

\begin{lem}\label{lem:Whitney:boundary}
Let $\Gamma\subsetneq\R^\dmn$ be a closed set. 
For each $R\in \mathcal{G}(\Gamma)$, let $\xi_R\in\Gamma$ be as in~\eqref{eqn:cube:point}, and let $\Delta_R$ be as in~\eqref{dfn:Delta:sub:Q}.

Then there is a $C$ depending only on  $\dmn$ such that, if $x\in\Gamma$ and $j\in \Z$, then
\begin{equation*}\{R\in \mathcal{G}(\Gamma):\ell(R)=2^j,\>x\in \Delta_R\}\end{equation*}
contains at most $C$ elements. 
\end{lem}
Recall that $\delta(y)=\dist(y,\Gamma)$.
By the bound~\eqref{eqn:Whitney:dilate:distance:max}, 
\begin{equation}\label{eqn:wdist}
\text{if }y\in 2Q\text{ then } \delta(y)\leq |y-\xi_Q|\le \frac{9}{2}\diam Q .\end{equation} 
In particular, 
if $y\in 2Q$ for some  $Q$ with $\ell(Q)\leq 2^j$ then $\delta(y)\leq 9\sqrt{\dmn} \cdot 2^{j-1}$. 
\subsection{Whitney decomposition and weighted averaged Sobolev spaces}
\begin{defn}\label{dfn:weighted:averaged}
Let $\Omega$ be an open set.
Recall that $\delta(y)=\dist(y,\partial\Omega)$.

Let $0<c<1$ and $1\leq \beta\leq\infty$. Let $0<p\leq\infty$ and let $s\in\R$.
The weighted averaged Lebesgue space ${L^{p}_{av,c,\beta,s}(\Omega)}$ is the set of all $f\in L^\beta_{loc}(\Omega)$ such that
\begin{equation*}\|f\|_{L^{p}_{av,c,\beta,s}(\Omega)}
= \biggl(\int_{\Omega}
\biggl(\fint_{B(y,c\delta(y))} |u|^\beta \biggr)^{p/\beta}
 \, \delta(y)^{\bdmn-\pdmn+p-ps}\,dy\biggr)^{1/p}\end{equation*}
is finite, with the usual modifications if $\beta=\infty$ or $p=\infty$.

The weighted averaged Sobolev space $\dot W^{1,p}_{av,c,\beta,s}(\Omega)$ is the set of all $f\in \dot W^{1,\beta}_{loc}(\Omega)$ such that
\begin{equation*}\|f\|_{\dot W^{1,p}_{av,c,\beta,s}(\Omega)}
= \|\nabla f\|_{L^p_{av,c,\beta,s}(\Omega)}
=
\biggl(\int_{\Omega}
\biggl(\fint_{B(y,c\delta(y))} |\nabla f|^\beta \biggr)^{p/\beta}\,
\delta(y)^{\bdmn-\pdmn+p-ps}\,dy\biggr)^{1/p}\end{equation*}
is finite.
\end{defn}

Whitney decomposition allows us to give an equivalent characterization of the weighted averaged $L^p$ norms. The next lemma was proven in \cite{BarMP26pA}.
\begin{lem}\label{lem:averaged:Whitney}Let $\Omega\subset\R^\dmn$ be a nonempty proper open subset of~$\R^\dmn$. Let $0<c<1$ and let $\mathcal{G}$ be as in Definition~\ref{dfn:Whitney}.

Let $
1\leq\beta\leq\infty$, $0< 
p\leq\infty$ and $s\in\R$. Let $F\in L^\beta_{loc}(\Omega)$. Then
\begin{equation}
\label{eqn:averaged:Whitney}
\|F\|_{L^p_{av,c,\beta,s}(\Omega)}
\approx
\biggl(\sum_{Q\in\mathcal{G}} \biggl(\fint_{2Q} |F|^\beta\biggr)^{p/\beta} \ell(Q)^{\bdmn+p-ps}\biggr)^{1/p}
\end{equation}
where the implicit constants depend only on $p$, $s$, $\dmn$, $\bdmn$, and~$c$.
\end{lem}
\begin{rmk}
Notice that, by the bound~\eqref{eqn:averaged:Whitney}, if $0<c<\Tilde{c}<1$ then 
\begin{equation}\label{eqn:equiv:c}
\|F\|_{L^p_{av,c,\beta,s}(\Omega)}\approx\|F\|_{L^p_{av,\Tilde{c},\beta,s}(\Omega)}.
\end{equation}
In light of this, we chose the value of $c=1/2$ to state our main theorems~\ref{thm:Dirichlet} and~\ref{thm:Dirichlet:Lq}. For this special value of $c$, we change notation:
\begin{equation*}
L^p_{av,\beta,s}(\Omega):=L^p_{av,1/2,\beta,s}(\Omega).
\end{equation*}
\end{rmk}
As an immediate corollary, we obtain that $L^p$ spaces are a special case of $ L^p$-averaged spaces.
\begin{cor}\label{cor:averaged:is:Wp}Let $\Omega\subset\R^\dmn$ be a nonempty open proper subset. Let $0<c<1$.
Let $0<p\leq\infty$ and $\theta\in\R$. Let $F\in L^p_{loc}(\Omega)$. Then
\begin{equation*}\int_\Omega \biggl(\fint_{B(x,c\delta(x))} |F|^p\biggr)\,\delta(x)^\theta\,dx
\approx
\int_\Omega |F|^p \delta^\theta
\end{equation*}
where the implicit constants depend only on $p$, $\theta$, $\dmn$, and~$c$.
\end{cor}

\section{Extensions}\label{sec:extension}

We now define our extension operator. This definition is based on that in \cite{DavFM21}.

\begin{defn}\label{dfn:extension}
Let $\Gamma\subset\R^\dmn$ be a closed nonempty  $\bdmn$-Ahlfors regular set for some $0<\bdmn<\dmn$. Let $g\in L^1_{loc}(\Gamma)$.

We define the extension operator $Eg$ as follows.
Let $\mathcal{G}$ be as in Definition~\ref{dfn:Whitney}. Let $\{\varphi_Q\}_{Q\in\mathcal{G}}$ be a smooth partition of unity such that $\varphi_Q$ is supported in $2Q$, $0\leq \varphi_Q\leq 1$, and $|\nabla\varphi_Q|\leq C(\diam Q)^{-1}$.

Now, for $y\in\R^\dmn\setminus\Gamma$ and for $g\in L^1_{loc}(\Gamma)$, let
\begin{equation}\label{eqn:extension}
Eg(y)=\sum_{Q\in\mathcal{G}} \varphi_Q(y) \fint_{\Delta_Q} g\,d\sigma,\end{equation}
where $\Delta_Q$ is as in~\eqref{dfn:Delta:sub:Q}.
\end{defn}
We now investigate the properties of this extension operator. 
As noted in Lemma~\ref{lem:dilate:Whitney}, if $R\in \mathcal{G}$ then there are at most $C_\dmn$ cubes $Q\in\mathcal{G}$  such that $2Q\cap 2R\neq \emptyset$. Thus $Eg\big\vert_{2R}$ is a linear combination of at most $C_\dmn$ functions $\varphi_Q$. Because the functions $\varphi_Q$ are smooth, their locally finite sums are smooth as well, and so
\begin{equation}
\label{eqn:E:smooth}\text{if }g\in L^1_{loc}(\Gamma)\text{ then }Eg\in C^\infty_{loc}(\R^\dmn\setminus\Gamma).\end{equation}

Furthermore, suppose $\Gamma$ is bounded. (By the Ahlfors condition, this is true if and only if $\sigma(\Gamma)<\infty$.) 
We claim that 
\begin{equation}\label{eqn:extension:far:constant}
Eg(y)=\fint_\Gamma g\,d\sigma\quad\text{for all $y\in\R^\dmn$ with $\dist(y,\Gamma)>\frac{9}{2} \diam \Gamma $.}\end{equation}
To see this, suppose that $y\in\R^\dmn$ with $\delta(y)>\frac{9}{2}\diam \Gamma $. If $Q\in\mathcal{G}$ and $y\in 2Q$, then, by~\eqref{eqn:wdist}, $\diam Q \geq \frac{2}{9}\delta(y)>\diam \Gamma$. Thus $\Delta_Q=\Gamma$ for all $Q\in\mathcal{G}$ with $2Q\owns y$. Because $\{\varphi_Q\}_{Q\in\mathcal{G}}$ is a partition of unity, we have that formula~\eqref{eqn:extension:far:constant} is valid.

\subsection{Extension of continuous functions}

In this subsection, we will show that if $g$ is continuous, then so is $Eg$. We will begin with the following localization lemma; that is, we will show that $Eg(y)$ depends only on the values of $g$ in the part of $\Gamma$ near~$y$.

\begin{lem}\label{lem:Eg:local}Suppose that $y\in \R^\dmn\setminus\Gamma$, that $x\in\Gamma$, and that $g\in L^1_{loc}(\Gamma)$. Then
\begin{equation*}Eg(y)=\sum_{Q\in\mathcal{G}} \varphi_Q(y) \frac{1}{\sigma(\Delta_Q)}\int_{\Delta_Q\cap \Delta(x,12|x-y|)} g\,d\sigma.\end{equation*}
In particular, if $x\in\Gamma$ satisfies $|x-y|=\dist(y,\Gamma)=\delta(y)$ then
\begin{equation*}Eg(y)=\sum_{Q\in\mathcal{G}} \varphi_Q(y) \frac{1}{\sigma(\Delta_Q)} \int_{\Delta_Q\cap \Delta(x,12\delta(y))} g\,d\sigma.\end{equation*}
\end{lem}
\begin{proof}
Recall that
\begin{equation*}Eg(y)=\sum_{Q\in\mathcal{G}} \varphi_Q(y) 
\frac{1}{\sigma(\Delta_Q)}\int_{\Delta_Q} g\,d\sigma.\end{equation*}
If $\varphi_Q(y)\neq 0$, then $y\in 2Q$. This means that
\begin{equation*}|x-y|\geq \dist(y,\Gamma)\geq \dist(2Q,\Gamma)\geq \frac{1}{2}\diam Q\end{equation*}
by Lemma~\ref{lem:dilate:Whitney}. Furthermore, for such~$Q$, by the bound~\eqref{eqn:Whitney:dilate:distance:max}
\begin{equation*}|y-\xi_Q|\leq
 \frac{9}{2}\diam Q \leq 9|x-y|.\end{equation*}
Thus, if $y\in\R^\dmn\setminus\Gamma$, $x\in\Gamma$, and $\varphi_Q(y)\neq 0$, then $\Delta_Q\subset\Delta(x,12|x-y|)$. Therefore, the lemma is proven, as restricting from $\Delta_Q$ to $\Delta_Q\cap\Delta(x,12|x-y|)$ only affects cubes $Q$ for which $\varphi_Q(y)=0$.
\end{proof}

\begin{cor}\label{cor:extension:of:cts}Suppose that that $x\in \Gamma$ and that $g:\Gamma\to\R$ is continuous at~$x$. Then 
\begin{equation*}\lim_{\substack{y\to x\\y\in\R^\dmn\setminus\Gamma}} Eg(y)=g(x).\end{equation*}
\end{cor}

\begin{proof}
Choose $\varepsilon>0$. By definition of continuity there is a $\delta>0$ such that if $\xi\in\Gamma$ and $|x-\xi|<\delta$ then $|g(x)-g(\xi)|<\varepsilon$.

By definition of~$E$, if $C$ is a constant, then $EC(y)=C$ for all $y\in\R^\dmn\setminus\Gamma$. In particular $|Eg(y)-g(x)|=|E(g-g(x))(y)|$.

If $y\in B(x,\delta/12)\setminus\Gamma$, then by Lemma~\ref{lem:Eg:local}
\begin{align*}|Eg(y)-g(x)|
&=|E(g-g(x))(y)|
\leq \sum_{Q\in\mathcal{G}} \varphi_Q(y) \frac{1}{\Delta_Q}\int_{\Delta_Q\cap \Delta(x,12|x-y|)} |g-g(x)|\,d\sigma
\\&< \sum_{Q\in\mathcal{G}} \varphi_Q(y) \,\varepsilon=\varepsilon
.\end{align*}
This completes the proof.
\end{proof}

\subsection{Local integrability of the extension operator and trace results}

Our next goal is to establish local integrability and a trace result for the extensions $Eg$. We begin with the following technical lemma, which contains the main computation; we will then use it to establish some results for $Eg$ near~$\Gamma$. We emphasize that, in all inequalities $I\le C II$, we will use the same constant $C$ to denote a quantity that depends on the fixed parameters specified in the statements of the Lemmas and Corollaries where the inequalities appear. Recall that $A$ is the Ahlfors regularity constant introduced in Definition~\ref{dfn:Ahlfors}.
\begin{lem}\label{lem:E:basic:bound}
Let $\Gamma\subset\R^\dmn$ be a closed nonempty  $\bdmn$-Ahlfors regular set for some $0<\bdmn<\dmn$.
If $k\in\Z$, $x\in\Gamma$, and $r<5\diam \Gamma $, define
\begin{equation}\label{eqn:V:ball}
\mathcal{V}_{k,x,r}=\{R\in\mathcal{G}{:} R\cap B(x,r)\neq\emptyset,\>\ell(R)=2^k\}.\end{equation}
If $g\in L^1_{loc}(\Gamma)$ and $E g$ is as in Definition~\ref{dfn:extension}, then we have the estimate
\begin{align*}\sum_{{R\in\mathcal{V}_{k,x,r}}}
\sup_R |Eg|^p\ell(R)^\theta
&\leq C
2^{k\theta-k\bdmn}
\int_{\Delta(x,K r)} 
|g(z)|^p\,d\sigma(z)
\end{align*}
whenever $1\leq p<\infty$ and $\theta\in\R$. Here $C>0$ and $K>1$ are constants depending only on $p$, $\theta$, $\dmn$, $\bdmn$, and $A$.
\end{lem}

\begin{proof}
By definition of $E$, and because $\varphi_Q=0$ in $R$ whenever $2Q\cap R=\emptyset$, we have that
\begin{align*}\sum_{{R\in\mathcal{V}_{k,x,r}}}
\sup_R |Eg|^p\ell(R)^\theta
&=
\sum_{{R\in\mathcal{V}_{k,x,r}}}
\sup_{y\in R} \biggl|\sum_{\substack{Q\in\mathcal{G}\\ 2Q\cap R\neq\emptyset}} \varphi_Q(y)\fint_{\Delta_Q} g\,d\sigma\biggr|^p
\ell(R)^\theta
.\end{align*}
Because $\sum_{\substack{Q\in\mathcal{G}}} \varphi_Q(y)=1$, we have that
\begin{align*}\sum_{{R\in\mathcal{V}_{k,x,r}}}
\sup_R |Eg|^p\ell(R)^\theta
&\leq
\sum_{{R\in\mathcal{V}_{k,x,r}}}
\sup_{\substack{Q\in\mathcal{G}\\ 2Q\cap R\neq\emptyset}}\biggl| \fint_{\Delta_Q} g\,d\sigma\biggr|^p
\ell(R)^\theta
.\end{align*}

If $Q\in\mathcal{G}$ and $2Q\cap R\neq\emptyset$, then by Lemma~\ref{lem:dilate:Whitney} we have that $\ell(Q)\approx\ell(R)=2^k$ and $\dist(Q,R)\leq C2^k$, and so $|\xi_Q-\xi_R|\leq C2^k$. Thus $\Delta_Q=\Delta(\xi_Q,\diam Q)\subset \Delta(\xi_R,C_12^k)$. 
Since $\ell(R)=2^k\le Cr<5C\diam \Gamma $, we may apply the Ahlfors regularity of the boundary to see that $\sigma(\Delta_Q)\approx \sigma(\Delta(\xi_R, C_12^k))$. Thus
\begin{align*}\sum_{{R\in\mathcal{V}_{k,x,r}}}
\sup_R |Eg|^p\ell(R)^\theta
&\leq C
\sum_{{R\in\mathcal{V}_{k,x,r}}}
\biggl(\fint_{\Delta(\xi_R,C_12^k)} |g|\,d\sigma\biggr)^p
\ell(R)^\theta
.\end{align*}
By H\"older's inequality
\begin{align*}\sum_{{R\in\mathcal{V}_{k,x,r}}}
\sup_R |Eg|^p\ell(R)^\theta
&\leq C
\sum_{{R\in\mathcal{V}_{k,x,r}}}
2^{k\theta-k\bdmn}
\int_{\Delta(\xi_R,C_1 2^k)} |g|^p\,d\sigma
.\end{align*}
We now claim that $\Delta(\xi_R,C_1 2^k)\subset \Delta(x,Kr)$ for some constant~$K$. Indeed, by~\eqref{eqn:R:B:size} we have $2^k=\ell(R)\le Cr$. 
Recalling that $\dist(\xi_R,R)\approx \ell(R)$ the claim follows.
Thus, changing the order of summation and integration, we have that
\begin{align*}\sum_{{R\in\mathcal{V}_{k,x,r}}}
\sup_R |Eg|^p\ell(R)^\theta
&\leq C
\sum_{{R\in\mathcal{V}_{k,x,r}}}
2^{k\theta-k\bdmn}
\int_{\Delta(x,Kr)}N(z)\, |g(z)|^p\,d\sigma(z)
\end{align*}
where $N(z)$ is the number of cubes $R\in\mathcal{V}_{k,x,r}$ such that $z\in \Delta(\xi_R,C_1 2^k)$.

If $z\in\Delta(\xi_R,C_1 2^k)$ then $\dist(z,R)\leq |z-\xi_R|+\dist(R,\Gamma)\leq C_1 2^k+C\ell(R)\le C\ell(R)$. Thus $R\subset B(z, C \ell(R))$ and by disjointness of Whitney cubes we can bound the number of such cubes~$R$. Thus
\begin{align*}\sum_{{R\in\mathcal{V}_{k,x,r}}}
\sup_R |E g|^p\ell(R)^\theta
&\leq C
2^{k\theta-k\bdmn}
\int_{\Delta(x,K r)} 
|g(z)|^p\,d\sigma(z)
\end{align*}
as desired. 
\end{proof}

We will need two corollaries to Lemma~\ref{lem:E:basic:bound}. The first states that $Eg$ is locally integrable in $\R^\dmn$ whenever $g$ is locally integrable on~$\Gamma$.

\begin{cor}\label{cor:E:L1}If $g\in L^1_{loc}(\Gamma)$, then $Eg\in L^1(B(x,r))$ for all $x\in\Gamma$ and all $r>0$. More precisely, we have that if $0<r<5\diam \Gamma $ then
\begin{equation*}\int_{B(x,r)}|Eg|
\leq 
Cr^{\pdmn-\pbdmn} \int_{\Delta(x, K r)} |g|\,d\sigma,\end{equation*}
where $C>0$ and $K>1$ are constants depending only on $\dmn$, $\bdmn$, and $A$.
\end{cor}
\begin{proof}
If $r\geq 5\diam \Gamma $ then we may bound $Eg$ using the local estimate~\eqref{eqn:extension:far:constant} outside of $B(x,5\diam \Gamma )$, and so we may assume that $r<5\diam \Gamma $.
We have that
\begin{equation*}\int_{B(x,r)}|Eg|
\leq \sum_{\substack{R\in\mathcal{G}\\R\cap B(x,r)\neq\emptyset}}\int_R |Eg|.\end{equation*}
Let $\mathcal{V}_{k,x,r}$ be as in~\eqref{eqn:V:ball}.
 By the bound~\eqref{eqn:R:B:size}, $\mathcal{V}_{k,x,r}=\emptyset$ if $2^k>Cr$, so 
\begin{equation*}\int_{B(x,r)}|Eg|
\leq \sum_{k=-\infty} ^{\lfloor \log_2 r\rfloor+C} \sum_{R\in \mathcal{V}_{k,x,r}}\int_R |Eg|,\end{equation*}
and by Lemma~\ref{lem:E:basic:bound} we have that
\begin{equation*}\int_{B(x,r)}|Eg|
\leq 
C\sum_{k=-\infty} ^{\lfloor \log_2 r\rfloor+C} 
2^{k\pdmn-k\pbdmn} \int_{\Delta(x,Kr)} |g|\,d\sigma
.\end{equation*}
Because $\bdmn<\dmn$, we may sum to see that
\begin{equation*}\int_{B(x,r)}|Eg|
\leq 
Cr^{\pdmn-\pbdmn} \int_{\Delta(x,Kr)} |g|\,d\sigma
<\infty\end{equation*}
as desired.
\end{proof}

Our next result establishes the trace of our extension operator (showing that it is, indeed, an extension operator).
The result~\eqref{eqn:trace:of:extension} is \cite[formula~(7.8)]{DavFM21}. The authors proved it for $\bdmn<\dmnMinusOne$ and $\Gamma$ unbounded, but the proof in the full generality we consider is the same. We include it here for completeness.
\begin{cor}\label{cor:trace:of:extension}
If $g\in L^1_{loc}(\Gamma)$, then for $\sigma$-almost every $x\in\Gamma$ we have that
\begin{equation}\label{eqn:trace:of:extension}\lim_{r\to 0^+} \fint_{B(x,r)} Eg = g(x)\end{equation}
and, in fact, have the  stronger convergence result
\begin{equation}\label{eqn:trace:of:extension:stronger}\lim_{r\to 0^+} \fint_{B(x,r)} |Eg-g(x)| = 0.\end{equation}
\end{cor}

\begin{proof}
Fix some $x\in\Gamma$ and let $G(z)=g(z)-g(x)$. By Corollary~\ref{cor:E:L1}, we have that if $0<r<5\diam \Gamma $ then
\begin{equation*}\int_{B(x,r)}|EG|
\leq 
Cr^{\pdmn-\pbdmn} \int_{\Delta(x,Kr)} |G|\,d\sigma.\end{equation*}
By Definition~\ref{dfn:extension} (the definition of our extension), and because $g(x)$ is a constant, we have that $EG(y)=Eg(y)-E(g(x))(y)=Eg(y)-g(x)$. Thus, applying the Ahlfors regularity of~$\Gamma$, if $r<\diam \Gamma $ then
\begin{equation*}\fint_{B(x,r)}|Eg-g(x)|
\leq 
C\fint_{\Delta(x,Kr)} |g-g(x)|\,d\sigma.\end{equation*}
Again, by the Ahlfors regularity condition, $\sigma$ is a doubling measure, and so by the Lebesgue differentiation theorem in doubling spaces (see, for example, \cite[Sections 2.8--2.9]{Fed69}), if $g\in L^1_{loc}(\Gamma)$ then $\sigma$-almost every $x\in \Gamma$ is a Lebesgue point for~$g$. By the definition of Lebesgue point, this completes the proof.
\end{proof}

\subsection{The gradient of the extension operator}

We now investigate the gradient of $Eg$ for $g$ lying in a Slobodeckij space $\dot \Lambda^{s,p}(\Gamma)$. As we did when investigating local integrability, we will present the basic technical computation first and then present the consequence for $\nabla E g$ as a corollary.

The $s=1/p=1/2$, $\bdmn<\dmnMinusOne$ (and $\Gamma$ unbounded) case of the bound~\eqref{eqn:E:Sobolev:bound} below, which is the main quantitative bound on $\nabla E g$ in $\R^\dmn\setminus\Gamma$, is \cite[Theorem~7.3]{DavFM21}. The proof of our generalization will closely follow Chapter~7 in \cite{DavFM21}.

\begin{lem}\label{lem:extension:gradient:morebasic}Let $\Gamma\subset\R^\dmn$ be a closed nonempty  $\bdmn$-Ahlfors regular set for some $0<\bdmn<\dmn$. Let $g\in L^1_{loc}(\Gamma)$ and let $Eg$ be as in Definition~\ref{dfn:extension}.
If $R\in \mathcal{G}$, then
\begin{align*}
\sup_{R}
|\nabla Eg|
&\leq 
\frac{C}{\ell(R)}
\fint_{\Delta_R}
\fint_{\Delta(\xi_R,K\ell(R))}  |g(x)-g(z)|\,d\sigma(z)\,d\sigma(x)
,\end{align*}
where $C>0$ and $K>1$ are constants depending only on $\dmn$, $\bdmn$, and $A$. 
\end{lem}
\begin{proof}
By formula~\eqref{eqn:extension:far:constant}, $Eg$ is constant on all cubes $R\in \mathcal{G}$ with $\ell(R)\geq C\diam \Gamma $, and so we may assume without loss of generality that $\ell(R)\leq C\diam \Gamma $.

Recall that (as observed in formula~\eqref{eqn:E:smooth}), if $g\in L^1_{loc}(\Gamma)$ then 
$Eg$ is smooth in~$\R^\dmn\setminus\Gamma$. Thus, the gradient exists in the classical sense. In fact, $Eg$ is a locally finite sum of smooth functions, and so if $y\in\R^\dmn$ then 
\begin{equation*}\nabla Eg(x)=\sum_{Q\in\mathcal{G}} \nabla \varphi_Q(x) \biggl(\fint_{\Delta_Q} g\,d\sigma\biggr).\end{equation*}

Observe that $\sum_{Q\in\mathcal{G}}\nabla\varphi_Q(y)=\nabla 1(y)=0$, and so $\sum_{Q\in\mathcal{G}} f(y) \nabla \varphi_Q(y) =0$ for any $f:\R^\dmn\to\R$. 
We therefore have that
\begin{equation*}
\nabla Eg(y)=\sum_{Q\in\mathcal{G}} \nabla \varphi_Q(y) \biggl(\fint_{\Delta_Q} g\,d\sigma-f(y)\biggr).\end{equation*}
We define $f(y)$ to be $\fint_{\Delta_R} g$ whenever $y\in R\in\mathcal{G}$. Thus, if $R\in\mathcal{G}$ then
\begin{equation*}\sup_{R}
|\nabla E g|
\leq \sum_{\substack{Q\in\mathcal{G}  \\ 2Q\cap R\neq\emptyset}} \sup_{R}|\nabla \varphi_Q | \biggl|\fint_{\Delta_Q} g\,d\sigma-\fint_{\Delta_R} g\,d\sigma\biggr|
.\end{equation*}
As usual $\{Q\in\mathcal{G}  : 2Q\cap R\neq\emptyset\}$ contains at most $C_{\dmn}$ elements all of diameter comparable to $\diam R $. Thus $|\nabla \varphi_Q (x)|\leq C/\ell(R)$ and $\Delta_Q\subset \Delta(\xi_R,K\ell(R))$ for all such~$Q$. Thus
\begin{align}
\sup_{R}
|\nabla E g|&\leq \frac{C}{\ell(R)}
\sum_{\substack{Q\in\mathcal{G}  \\ 2Q\cap R\neq\emptyset}} 
\biggl|\fint_{\Delta_Q} g\,d\sigma-\fint_{\Delta_R} g\,d\sigma\biggr|\notag
\\&= \frac{C}{\ell(R)}
\sum_{\substack{Q\in\mathcal{G}  \\ 2Q\cap R\neq\emptyset}} 
\biggl|\fint_{\Delta_R} \fint_{\Delta_Q}  g(x)-g(z)\,d\sigma(z)\,d\sigma(x)\biggr|\notag
\\&\leq \frac{C}{\ell(R)}
\fint_{\Delta_R}
\fint_{\Delta(\xi_R,K\ell(R))}  |g(x)-g(z)|\,d\sigma(z)\,d\sigma(x),\label{eqn:grad:ext:bd:whitney}
\end{align}
as desired.
\end{proof}

\begin{lem}\label{lem:extension:gradient:basic}
Let $\Gamma\subset\R^\dmn$ be a closed nonempty  $\bdmn$-Ahlfors regular set for some $0<\bdmn<\dmn$. Let $1\leq p\leq\infty$, let $-\pbdmn/p<s<\infty$, and let $g\in \dot \Lambda^{s,p}(\Gamma)\subset L^1_{loc}(\Gamma)$.
Let $Eg$ be as in Definition~\ref{dfn:extension}.

If $J$ is an integer, define 
\begin{equation*}S(J)=
\biggl(\sum_{\substack{R\in\mathcal{G} \\ \ell(R)\leq 2^J}}
\Bigl(\sup_R |\nabla Eg|\Bigr)^{p}
\ell(R)^{\bdmn+p-ps}\biggr)^{1/p}
\end{equation*}
with the obvious modification in the case $p=\infty$.

Then
\begin{equation*}S(J)\leq C\biggl(\int_\Gamma \int_{\Delta(x,K_1 2^J)} \frac{|g(x)-g(z)|^p}{|x-z|^{\bdmn+ps} }\,d\sigma(z)\,d\sigma(x)
\biggr)^{1/p},\end{equation*}
where $C>0$ and $K_1>1$ are constants depending only on $p$, $s$, $\dmn$, $\bdmn$, and $A$. 
\end{lem}

\begin{proof}
By formula~\eqref{eqn:extension:far:constant}, $Eg$ is constant on all cubes $R\in \mathcal{G} $ with $\ell(R)\geq C\diam \Gamma $, and so we may assume without loss of generality that $2^J\leq C\diam \Gamma $.

If $p=\infty$, then we have that
\begin{equation*}S( J)=\sup_{\substack{R\in\mathcal{G}  \\ \ell(R)\leq 2^J}} \sup_{R}
|\nabla E g| \,\ell(R)^{1-s}\end{equation*}
which, by Lemma~\ref{lem:extension:gradient:morebasic}, satisfies
\begin{align*}
S( J)
&\leq
\sup_{\substack{R\in\mathcal{G}  \\ \ell(R)\leq 2^J}}
\frac{C}{\ell(R)^s}
\fint_{\Delta_R} \fint_{\Delta(\xi_R,K\ell(R))}  |g(x)-g(z)|\,d\sigma(z)\,d\sigma(x)
.\end{align*}
Observe that $|x-z|\leq K_1\ell(R)\leq K_12^J$, where $K_1=K+\sqrt{n}$. Thus
\begin{align*}
S( J)
&\leq
\sup_{\substack{R\in\mathcal{G}  \\ \ell(R)\leq 2^J}} \frac{C}{\ell(R)^s}
\sup_{\substack{x,z\in\Gamma\\|x-z|\leq K_1\ell(R)}}
|g(x)-g(z)|
\leq
C
\sup_{\substack{x,z\in\Gamma\\|x-y|\leq K_12^J}}
\frac{|g(x)-g(z)|}{|x-z|^s}
\end{align*}
which is the desired estimate.

Now suppose $1\leq p<\infty$.
Applying H\"older's inequality and Lemma~\ref{lem:extension:gradient:morebasic}, we see that 
\begin{align*}
\sup_{R}
|\nabla E g|
&\leq 
\frac{C}{\ell(R)}
\biggl(
\fint_{\Delta_R} \fint_{\Delta(\xi_R,K\ell(R))}  |g(x)-g(z)|^p\,d\sigma(z)\,d\sigma(x)\biggr)^{1/p}
.\end{align*}
If $x\in \Delta_R$, then $|x-\xi_R|<\diam(R)$ and so $\Delta(\xi_R,K\ell(R))\subset \Delta(x,K_1\ell(R))$.
Recall that we have assumed that $\ell(R)\leq 2^J\leq C\diam \Gamma $. We may thus apply the Ahlfors condition to see that
\begin{align*}
\sup_{R}
|\nabla E g|
&\leq 
\frac{C}{\ell(R)^{1+2\pbdmn/p}}
\biggl(
\int_{\Delta_R} \int_{\Delta(x,K_1\ell(R))}  |g(x)-g(z)|^p\,d\sigma(z)\,d\sigma(x)\biggr)^{1/p}
.\end{align*}
Thus
\begin{align*}S( J)^p
&=
\sum_{\substack{R\in\mathcal{G}  \\ \ell(R)\leq 2^J}}
\Bigl(\sup_R |\nabla E g|\Bigr)^{p}
\ell(R)^{\bdmn+p-ps}
\\&\leq
\sum_{\substack{R\in\mathcal{G}  \\ \ell(R)\leq 2^J}}
\frac{C}{\ell(R)^{ps+\pbdmn}}
\int_{\Delta_R} \int_{\Delta(x,K_1\ell(R))}  |g(x)-g(z)|^p\,d\sigma(z)\,d\sigma(x)
.\end{align*}
Interchanging the order of summation and integration, we see that
\begin{align*}
S( J)^p
&\leq 
\int_{\Gamma} \sum_{\substack{R\in\mathcal{G}  \\ \ell(R)\leq 2^J \\ x\in \Delta_R}} 
\frac{C}{\ell(R)^{\bdmn+ps}}
 \int_{\Delta(x,K_1\ell(R))} |g(x)-g(z)|^p\,d\sigma(z)\,d\sigma(x)
.\end{align*}
By Lemma~\ref{lem:Whitney:boundary}, if $x\in\Gamma$ and $k\in\Z$, then there are at most $C$ cubes $R\in\mathcal{G} $ with $x\in \Delta_R$ and $\ell(R)=2^k$. Thus,
\begin{align*}
S( J)^p
&\leq 
\int_{\Gamma} 
\sum_{ k=-\infty}^J
\sum_{\substack{R\in\mathcal{G}  \\ \ell(R)= 2^k \\ x\in \Delta_R}} 
\frac{C}{\ell(R)^{\bdmn+ps}}
 \int_{\Delta(x,K_1\ell(R))} |g(x)-g(z)|^p\,d\sigma(z)\,d\sigma(x)
\,
\\&\leq C
\int_{\Gamma} 
\sum_{ k=-\infty}^J
2^{-k(\bdmn+ps)}
\int_{\Delta(x,K_12^k)} |g(x)-g(z)|^p\,d\sigma(z)\,d\sigma(x)
.\end{align*}
We change the order of integration and summation again to see that
\begin{align*}
S( J)^p
&\leq C
\int_{\Gamma} \int_{\Delta(x,K_12^J)} 
\sum_{\substack{ k\leq J\\ 2^k>|x-z|/K_1}}
2^{-k(\bdmn+ps)}
|g(x)-g(z)|^p\,d\sigma(z)\,d\sigma(x)
.\end{align*}
Because $\bdmn+ps>0$, the sum converges and 
\begin{align*}
S(J)^p
&\leq C
\int_{\Gamma} \int_{\Delta(x,K_12^J)} 
|x-z|^{-\bdmn-ps}
|g(x)-g(z)|^p\,d\sigma(z)\,d\sigma(x)
\end{align*}
as desired.
\end{proof}

We now establish  estimates for the $\dot W^{1,p}_{av,c,\beta,s}(\R^\dmn\setminus\Gamma)$ norm of~$Eg$.

\begin{cor}\label{cor:E:Sobolev:bound}
Let $\Gamma\subset\R^\dmn$ be a closed nonempty  $\bdmn$-Ahlfors regular set for some $0<\bdmn<\dmn$. Let $1\leq p\leq\infty$, $-\bpdmn/p<s<\infty$, $0<c<1$, and $g\in \dot \Lambda^{s,p}(\Gamma)\subset L^1_{loc}(\Gamma)$.
Let $Eg$ be as in Definition~\ref{dfn:extension}.
Then, for any $1\le\beta\le\infty$, we have the bound
\begin{align}
\label{eqn:E:Sobolev:bound}
\|Eg\|_{\dot W^{1,p}_{av,c,\beta,s}(\R^\dmn\setminus\Gamma)}
&\leq
C \|g\|_{\dot \Lambda^{s,p}(\Gamma)},
\end{align}
where $C$ is a constant depending on $p$, $s$, $\dmn$, $\bdmn$, $c$, and $A$.
\end{cor}

\begin{proof} We present only the case $p<\infty$; the $p=\infty$ case follows with suitable modifications.
It is clearly enough to prove~\eqref{eqn:E:Sobolev:bound} for $\beta=\infty$; the general case is a consequence of the Hölder inequality.
By Lemma~\ref{lem:averaged:Whitney}, we have that 
\begin{align*}
\|Eg\|_{\dot W^{1,p}_{av,c,\infty,s}(\R^\dmn\setminus\Gamma)}
&\approx
\biggl(\sum_{R\in\mathcal{G}}
\Bigl(\sup_R |\nabla Eg|\Bigr)^{p}
\ell(R)^{\bdmn+p-ps}\biggr)^{1/p}
.\end{align*}
By Lemma~\ref{lem:extension:gradient:basic}, we have that
\begin{equation*}
\sum_{\substack{R\in\mathcal{G}  \\ \ell(R)\leq 2^J}}
\Bigl(\sup_R |\nabla E g|\Bigr)^{p}
\ell(R)^{\bdmn+p-ps}
\leq 
C\int_\Gamma \int_{\Delta(x,K_1 2^J)} \frac{|g(x)-g(z)|^p}{|x-z|^{\pbdmn+ps} } \,d\sigma(z)\,d\sigma(x).
\end{equation*}
Letting $J\to\infty$ and recalling from Definition~\ref{dfn:Slobodeckij} that
\begin{equation*}\|g\|_{\dot \Lambda^{p,s}(\Gamma)}^p=
\int_{\Gamma} \int_{\Gamma} 
\frac{|g(x)-g(z)|^p}{|x-z|^{\pbdmn+ps} } \,d\sigma(z)\,d\sigma(x)
\end{equation*}
completes the proof.
\end{proof}
\subsection{Extensions for $p<1$}

In this section, we investigate extensions of functions in the atomic space $\dot {A} ^ {s,p}(\Gamma)$ for $0<s<1$ and $p\leq 1$.

\begin{prp}
Let $\Gamma\subset\R^\dmn$ be a closed $\bdmn$-Ahlfors regular set for some $0<\bdmn<\dmn$. Let $z\in\Gamma$, $0<\varrho<2\diam\Gamma$ and let $a$ be a block in $\Delta(z,\varrho)$ in the sense of Definition~\ref{dfn:atom}. Let $0<s<1$ and let $0<p\leq 1$ be such that $p> \bdmn/(\bdmn+s)$.

If $0<c<1$ and $1\leq \beta\leq \infty$, then we have the estimate
\begin{equation*}
\|Ea\|_{\dot W^{1,p}_{av,c,\beta,s}(\R^\dmn\setminus\Gamma)}
\leq C\varrho^{\bdmn/p-\bdmn /2-s}.
\end{equation*}

\end{prp}

\begin{proof}
By Lemma~\ref{lem:averaged:Whitney}, 
\begin{equation*}\|Ea\|_{\dot W^{1,p}_{av,c,\beta,s}(\R^\dmn\setminus\Gamma)}^p
\approx \sum_{R\in\mathcal{G}} \biggl(\fint_R |\nabla Ea|^\beta\biggr)^{p/\beta} \ell(R)^{\bdmn+p-ps}.\end{equation*}

By the definition~\eqref{eqn:extension} of~$Ea$, if $Ea(y)\neq 0$ then there is at least one $Q\in\mathcal{G}$ with $y\in 2Q$ and with $a\neq 0$ somewhere in $\Delta_Q$. 
By Definition~\ref{dfn:atom}, this means that $\Delta(z,\varrho)\cap \Delta_Q\neq \emptyset$, and thus $|z-\xi_Q|<\varrho+\diam Q$.

Thus, if $Ea\not\equiv 0$ in $R$ for some $R\in\mathcal{G}$, then $R\cap 2Q\neq \emptyset$ for some $Q\in\mathcal{G}$ with $|z-\xi_Q|<\varrho+\diam Q$. By Lemma~\ref{lem:dilate:Whitney}, $\diam Q\leq 9\diam R$ and $\dist(\xi_Q,R)\leq \sup_{x\in 2Q} |x-\xi_Q|\leq \frac{9}{2}\diam Q$, and so
\begin{equation*}\dist(z,R)<\varrho+\frac{11}{2}\diam Q\leq \varrho+\frac{99}{2}\diam R.\end{equation*}

Thus
\begin{equation*}\|Ea\|_{\dot W^{1,p}_{av,c,\beta,s}(\R^\dmn\setminus\Gamma)}^p
\approx \sum_{\substack{R\in\mathcal{G}\\\dist(z,R)\leq \varrho+C_3\diam R}} \biggl(\fint_R |\nabla Ea|^\beta\biggr)^{p/\beta} \ell(R)^{\bdmn+p-ps}.\end{equation*}
Let $J=\infty$ if $\Gamma$ is unbounded, and let $J$ satisfy $2^{J}\leq 9\diam\Gamma< 2^{J+1}$ otherwise. By formula~\eqref{eqn:extension:far:constant} and the definition of Whitney decomposition, $Ea$ is constant in $R$ whenever $\ell(R)>2^J$. Replacing the average by a supremum and summing over the possible side lengths of~$R$, we see that
\begin{equation*}
\|Ea\|_{\dot W^{1,p}_{av,c,\beta,s}(\R^\dmn\setminus\Gamma)}^p
\leq C \sum_{k=-\infty}^J \sum_{\substack{R\in\mathcal{G} \\\ell(R)=2^k\\\dist(z,R)\leq \varrho+C_3\diam R}} \bigl(\sup_R |\nabla Ea|\bigr)^{p} 2^{k(\bdmn+p-ps)}
.\end{equation*}
For notational convenience let $\mathcal{G}_{k,z,\varrho}=\{R\in\mathcal{G} :\ell(R)=2^k$ and $\dist(z,R)\leq \varrho+C_3\diam R\}$.
Applying Lemma~\ref{lem:extension:gradient:morebasic}, we see that 
\begin{multline*}
\|Ea\|_{\dot W^{1,p}_{av,c,\beta,s}(\R^\dmn\setminus\Gamma)}^p
\\\leq C\sum_{k=-\infty}^J \sum_{R\in\mathcal{G}_{k,z,\varrho}} \biggl(
\fint_{\Delta_R}
\fint_{\Delta(\xi_R,C_2\ell(R))}  |a(x)-a(y)|\,d\sigma(y)\,d\sigma(x)\biggr)^{p} 2^{k(\bdmn-ps)}
.\end{multline*}

By Definition~\ref{dfn:atom}, for any $R\in\mathcal{G}$ we have that
\begin{multline*}
\fint_{\Delta_R}
\fint_{\Delta(\xi_R,C_2\ell(R))}  |a(x)-a(y)|\,d\sigma(y)\,d\sigma(x)
\\\leq
\varrho^{-\bdmn/2-1}
\fint_{\Delta_R}
\fint_{\Delta(\xi_R,C_2\ell(R))}  |x-y|\,d\sigma(y)\,d\sigma(x)
\leq
C\varrho^{-\bdmn/2-1}\ell(R)
.\end{multline*}
Furthermore, if $R\in\mathcal{G}_{k,z,\varrho}$ with $\ell(R)\leq\varrho$, then $R\subset \{x\in B(z,C\varrho):\dist(x,\Gamma)\leq C2^k\}$. By the bound~\eqref{eqn:close:to:Ahlfors}, and because Whitney cubes are pairwise disjoint, there are at most
$C\varrho^\bdmn (2^k)^{-\pbdmn}$ such cubes. Thus
\begin{multline*}\smash{\sum_{\substack{k\leq J\\2^k\leq\varrho}}} \sum_{R\in\mathcal{G}_{k,z,\varrho}} \biggl(
\fint_{\Delta_R}
\fint_{\Delta(\xi_R,C_2\ell(R))}  |a(x)-a(y)|\,d\sigma(y)\,d\sigma(x)\biggr)^{p} 2^{k(\bdmn-ps)}
\\\leq
C\varrho^{-\bdmn p/2-p}
\sum_{{2^k\leq\varrho}} 
2^{k(\bdmn+p-ps)}
\leq
C\varrho^{\bdmn-\bdmn p/2-ps}
.\end{multline*}

Now suppose that $\varrho<\ell(R)$ and that $R\in \mathcal{G}_{k,z,\varrho}$ for some $k\leq J$. 
Because $k\leq J$ we have that $\ell(R)<C\diam\Gamma$. By the Ahlfors condition
\begin{multline*}
\fint_{\Delta_R}
\fint_{\Delta(\xi_R,C_2\ell(R))}  |a(x)-a(y)|\,d\sigma(y)\,d\sigma(x)
\\\leq
\frac{C}{\ell(R)^{2\bdmn}}
\int_{\Delta_R}
\int_{\Delta(\xi_R,C_2\ell(R))}  
|a(x)-a(y)|\,d\sigma(y)\,d\sigma(x)
.\end{multline*}
Let $\widetilde \Delta_R=\Delta(\xi_R,C_2\ell(R))$ and let $\Delta_a=\Delta(z,\varrho)$.
We have that
\begin{align*}
\Delta_R\times \Delta(\xi_R,C_2\ell(R))
&=\Delta_R\times \widetilde \Delta_R
\\&= \bigl[(\Delta_R\cap\Delta_a)\times (\widetilde\Delta_R\cap\Delta_a)\bigr] \bigcup 
\bigl[(\Delta_R\setminus\Delta_a)\times(\widetilde\Delta_R\cap\Delta_a) \bigr]
\\&\qquad\bigcup 
\bigl[(\Delta_R\cap\Delta_a)\times(\widetilde\Delta_R\setminus\Delta_a) \bigr]\bigcup 
\bigl[(\Delta_R\setminus\Delta_a)\times(\widetilde\Delta_R\setminus\Delta_a) \bigr].\end{align*}
If $(x,y)\in (\Delta_R\setminus\Delta_a)\times(\widetilde\Delta_R\setminus\Delta_a) $, then $a(x)=a(y)=0$ because $a$ is supported in $\Delta_a$. Furthermore, $\Delta_R\subset \widetilde \Delta_R$. Thus
\begin{align*}\int_{\Delta_R}
\int_{\widetilde\Delta_R}  
|a(x)-a(y)|\,d\sigma(y)\,d\sigma(x)
&\leq
\int_{\Delta_a}
\int_{\Delta_a}  
|a(x)-a(y)|\,d\sigma(y)\,d\sigma(x)
\\&\qquad+
2\int_{\Delta_a}
\int_{\widetilde\Delta_R\setminus\Delta_a}  
|a(x)-a(y)|\,d\sigma(y)\,d\sigma(x)
.\end{align*}
Again taking estimates on~$a$ from Definition~\ref{dfn:atom}, we have that
\begin{align*}\int_{\Delta_R}
\int_{\widetilde\Delta_R}  
|a(x)-a(y)|\,d\sigma(y)\,d\sigma(x)
&\leq
\int_{\Delta_a}
\int_{\Delta_a}  
2\varrho^{-\bdmn/2}\,d\sigma(y)\,d\sigma(x)
\\&\qquad+
2\int_{\widetilde\Delta_R\setminus\Delta_a}
\int_{\Delta_a}  
\varrho^{-\bdmn/2}\,d\sigma(y)\,d\sigma(x)
.\end{align*}
Applying the Ahlfors condition again and recalling $\ell(R)>\varrho$, we see that
\begin{align*}\int_{\Delta_R}
\int_{\widetilde\Delta_R}  
|a(x)-a(y)|\,d\sigma(y)\,d\sigma(x)
&\leq
C\ell(R)^\bdmn
\varrho^{\bdmn/2}
.\end{align*}
Finally, because Whitney cubes are disjoint, if $2^k>\varrho$ then there are at most $C$ cubes $R\in \mathcal{G}$ with $\ell(R)=2^k$ and $\dist(R,z)\leq \varrho+C_3\diam R\leq C2^k$.
Thus
\begin{multline*}
\smash{\sum_{\substack{ k\leq J\\2^k>\varrho}}} \sum_{R\in\mathcal{G}_{k,z,\varrho}} \biggl(
\fint_{\Delta_R}
\fint_{\Delta(\xi_R,C_2\ell(R))}  |a(x)-a(y)|\,d\sigma(y)\,d\sigma(x)\biggr)^{p} 2^{k(\bdmn-ps)}
\\\leq
C\varrho^{p\bdmn/2}
{\sum_{{2^k>\varrho}}} 2^{k(\bdmn-p\bdmn-ps)}
.\end{multline*}
If $p> \bdmn/(\bdmn+s)$, then $\bdmn-p\bdmn-ps<0$ and so the geometric series converges, and 
\begin{multline*}
\smash{\sum_{\substack{ k\leq J\\2^k>\varrho}}} \sum_{R\in\mathcal{G}_{k,z,\varrho}} \biggl(
\fint_{\Delta_R}
\fint_{\Delta(\xi_R,C_2\ell(R))}  |a(x)-a(y)|\,d\sigma(y)\,d\sigma(x)\biggr)^{p} 2^{k(\bdmn-ps)}
\\\leq
C\varrho^{\bdmn-p\bdmn/2-ps}
.\end{multline*}
This completes the proof.

\end{proof}
As an immediate corollary, we have gradient bounds of the extension operator in~$\dot {A} ^ {s,p}$.
\begin{cor}\label{cor:atom:ext:bound}
Let $\Gamma\subset\R^\dmn$ be a closed $\bdmn$-Ahlfors regular set for some $0<\bdmn<\dmn$. Let $0<s<1$ and let $0<p\leq 1$ be such that $p> \bdmn/(\bdmn+s)$.

Then, if $0<c<1$ and $1\leq \beta\leq \infty$, for any $g\in \dot A^{s,p}$ we have the estimate
\begin{equation}\label{eqn:atom:ext:bound}
\|Eg\|_{\dot W^{1,p}_{av,c,\beta,s}(\R^\dmn\setminus\Gamma)}
\leq C\|g\|_{\dot A^{s,p}},
\end{equation}
where $C$ is a constant depending on $p$, $s$, $\dmn$, $\bdmn$, $c$, and $A$.

\end{cor}

\section{Poisson implies Dirichlet}\label{sec:pf:thm}
In this section, we provide a proof of Theorems~\ref{thm:Dirichlet} and~\ref{thm:Dirichlet:Lq}. We will use the extension machinery to see how  well posedness of the Poisson problem implies  well posedness of the Poisson-Dirichlet problem. Here  $\Omega$ will be a domain with Ahlfors regular boundary of dimension $\bdmn$, and $L$ will be an elliptic operator as defined in~\eqref{eq:operator}, satisfying~\eqref{eqn:elliptic:introduction} and~\eqref{eqn:elliptic:introduction:1}.
Let us first give our definition of  well posedness. 
 \begin{defn}
 Let $0<s<1$ and $\pbdmn/(\bdmn+s)<p\le\infty$. We say that the $L^{p}_{av,\beta,s}(\Omega)$-Poisson problem for $L$ is  well posed in~$\Omega$ if, for every $\vec H\in L^{p}_{av,\beta,s}(\Omega)$, there is a unique function $u$ that satisfies
\begin{equation}\label{poisson:equation}
\left\{\begin{aligned}
Lu&=-\Div(\delta^{\bdmn+1-\pdmn}\vec H) \text{ in } \Omega,\\
u&\in  \dot W^{1,p}_{av,\beta,s}(\Omega),\\
\Tr& u=0,
\end{aligned}\right.\end{equation}
and furthermore for some (hence every, by~\eqref{eqn:equiv:c}) $c\in (0,1)$ there exists a constant $C$ (depending on~$c$) such that 
\begin{equation}\label{poisson:estimate}
\|u\|_{\dot W^{1,p}_{av,c,\beta,s}(\Omega)}
\leq C\|\vec H\|_{L^p_{av,c,\beta,s}(\Omega)}.
\end{equation}

We say that the $\dot B^{s,p}_p(\partial\Omega)$-Dirichlet problem for $L$ is  well posed in~$\Omega$ if, for every $f\in \dot B^{s,p}_p(\partial\Omega)$, there is a unique function $u$ that satisfies
\begin{equation}\label{dirichlet:equation}
\left\{\begin{aligned}
Lu&=0 \text{ in } \Omega,\\
u&\in  \dot W^{1,p}_{av,\beta,s}(\Omega),\\
\Tr& u=f,
\end{aligned}\right.\end{equation}
and furthermore for some $c\in (0,1)$ there exists a constant $C$ (depending on~$c$) such that
\begin{equation}\label{dirichlet:estimate}
\|u\|_{\dot W^{1,p}_{av,c,\beta,s}(\Omega)}
\leq C\| f\|_{\dot B^{s,p}_p(\partial\Omega)}.
\end{equation}
 \end{defn}
 \begin{rmk}\label{final:remark}
 Assume that the $L^{p}_{av,\beta,s}(\Omega)$-Poisson problem and the $\dot B^{s,p}_p(\partial\Omega)$-Dirichlet problem for $L$ are both  well posed. Then, by linearity, for every $\vec H\in L^{p}_{av,\beta,s}(\Omega)$ and $f\in \dot B^{s,p}_p(\partial\Omega)$ there is a unique function $u$ that satisfies
 \begin{equation}\label{poisson-dirichlet:equation}
\left\{\begin{aligned}
Lu&=-\Div(\delta^{\bdmn+1-\pdmn}\vec H) \text{ in } \Omega,\\
u&\in  \dot W^{1,p}_{av,\beta,s}(\Omega),\\
\Tr& u=f,
\end{aligned}\right.\end{equation}
and furthermore for some $c\in (0,1)$ there exists a constant $C$ (depending on~$c$) such that
\begin{equation}\label{p:dirichlet:estimate}
\|u\|_{\dot W^{1,p}_{av,c,\beta,s}(\Omega)}
\leq C\|\vec H\|_{L^p_{av,c,\beta,s}(\Omega)}+C\| f\|_{\dot B^{s,p}_p(\partial\Omega)}.
\end{equation}
 \end{rmk}
 \begin{prp}\label{prp:extension:sol}
Let $\Omega$ be a domain with $\bdmn$-Ahlfors regular boundary, let $1\le\beta\le\infty$, $0<s<1$ and $\pbdmn/(\bdmn+s)<p\le\infty$. Assume that the $L^{p}_{av,\beta,s}(\Omega)$-Poisson problem for $L$ is  well posed in~$\Omega$. Then the $\dot B^{s,p}_p(\partial\Omega)$-Dirichlet problem for $L$ is  well posed in~$\Omega$.
 \end{prp}
 \begin{proof}
 Let $f\in \dot B^{s,p}_p(\partial\Omega)$ and set 
 \begin{equation*}
 \vec H:= -\delta^{-(1+\bdmn-\dmn)}A\nabla E f.
 \end{equation*}
 By Corollary~\ref{cor:E:Sobolev:bound} and Corollary~\ref{cor:atom:ext:bound} we have that $\vec H\in L^p_{av,\beta,s}(\Omega)$ and
 \begin{equation}\label{eqn:Htof}
 \|\vec H\|_{L^p_{av,1/2,\beta,s}(\Omega)}\le C\|f\|_{\dot B^{s,p}_p(\partial\Omega)},
 \end{equation}
 where we used~\eqref{eqn:elliptic:introduction:1}. Therefore, by assumption there is a solution $w$ to~\eqref{poisson:equation} satisfying~\eqref{poisson:estimate}, and therefore, in light of~\eqref{eqn:Htof},
 \begin{equation*}
 \|w\|_{\dot W^{1,p}_{av,1/2,\beta,s}(\Omega)}\le C \|f\|_{\dot B^{s,p}_p(\partial\Omega)}.
 \end{equation*}
 Set 
 \begin{equation*}
 u:=w+Ef.
 \end{equation*}
 By construction, we have 
 \begin{equation*}
 Lu=Lw-\Div(A\nabla Ef)=0
 \end{equation*}
and Corollary~\ref{cor:trace:of:extension} implies
 \begin{equation*}
 \Tr(u)=0+\Tr(Ef)=f.
 \end{equation*}
 Finally, by~\eqref{eqn:E:Sobolev:bound},~\eqref{eqn:atom:ext:bound} and~\eqref{eqn:Htof},
 \begin{align*}
 \|u\|_{\dot W^{1,p}_{av,1/2,\beta,s}(\Omega)}&\le C_p \|w\|_{\dot W^{1,p}_{av,1/2,\beta,s}(\Omega)}+C_p \|Ef\|_{\dot W^{1,p}_{av,1/2,\beta,s}(\Omega)}\\
 &\le C \|f\|_{\dot B^{s,p}_p(\partial\Omega)}.
 \end{align*}
Thus $u$ solves~\eqref{dirichlet:equation}. Assume that $v$ is another solution; then $u-v$ solves the Poisson problem~\eqref{dirichlet:equation} with $\vec {H} = 0$, and, by assumption, this implies $u=v$.
 \end{proof}
 
 As an immediate corollary, we obtain our main results.
 \begin{proof}[Proof of Theorems~\ref{thm:Dirichlet} and~\ref{thm:Dirichlet:Lq}]
 Let $(s,1/p)$ be as in Figure~\ref{fig:arbitrary} or~\ref{fig:extrapolation} and $p_{L^*}^-<\beta<p_L^+$. By \cite[Theorems 1.5 and 1.7]{BarMP26pA} the $L^p_{av,\beta,s}$-Poisson problem is  well posed. By Proposition \ref{prp:extension:sol}, the Dirichlet problem is well posed too, and Remark~\ref{final:remark} ends the proof.
 \end{proof}
 \begin{rmk}
 The gradient bounds of the extension operator (see Corollaries~\ref{cor:E:Sobolev:bound} and~\ref{cor:atom:ext:bound}) hold in the full generality of domains with Ahlfors regular boundaries. All the additional assumptions on $\Omega$ appearing in Theorems~\ref{thm:Dirichlet} and~\ref{thm:Dirichlet:Lq} are required by the Poisson version ($f=0$) of Theorem~\ref{thm:Dirichlet} that we proved in  \cite{BarMP26pA}.  
\end{rmk}

\def\cprime{'}
\bibliographystyle{amsalpha}
\bibliography{bibli}
\end{document}